\documentclass[reqno,11pt]{amsart}

\usepackage{amssymb}
\usepackage[all]{xy}
\usepackage{caption}
\usepackage[pdftex]{hyperref}
\usepackage{graphicx}
\usepackage{array}
\usepackage{physics}
\usepackage{color}
\usepackage{xy}
\usepackage{bm}
\usepackage{comment}
\usepackage{mathtools}

\allowdisplaybreaks

\newtheorem{thmIntro}{Theorem}
\newtheorem{thm}{Theorem}[section]

\newtheorem{prop}[thm]{Proposition}
\newtheorem{cor}[thm]{Corollary}
\theoremstyle{definition}
\newtheorem{defn}[thm]{Definition}
\newtheorem{ex}[thm]{Example}
\newtheorem{rem}[thm]{Remark}

\providecommand\sslash{\mathbin{/\mkern-5.5mu/}}

\newcommand{\sldos}{\SL_2(\mathbb{C})}

\newcommand{\LL}{{\mathbb L}}
\newcommand{\CC}{{\mathbb C}}
\newcommand{\ZZ}{{\mathbb Z}}
\newcommand{\RR}{{\mathbb R}}
\newcommand{\cB}{{\mathcal B}}

\newcommand{\FQ}{\textbf{FQ}}
\newcommand{\G}{{\Gamma}}
\newcommand{\cT}{{\mathcal{T}}}
\newcommand{\R}{{\mathcal{R}}}
\newcommand{\X}{{\mathcal{X}}}

\newcommand{\cA}{{\mathcal{A}}}

\newcommand{\x}{\times}

\newcommand{\Free}[1]{\mathbf{F}_{#1}}
\newcommand{\SVB}{\textbf{\textup{SVectB}}}

\newcommand{\Qdl}{\textbf{\textup{Qdl}}}

\newcommand{\AQ}[1]{{\textup{AQ}_{#1}}}
\newcommand{\GL}[1]{{\textup{GL}_{#1}}}
\newcommand{\AGL}[1]{{\textup{AGL}_{#1}}}

\newcommand{\aglC}{\AGL{1}(\mathbb{C})}
\newcommand{\agl}{\AGL{1}}

\DeclareMathOperator{\Spec}{Spec}
\DeclareMathOperator{\id}{Id}
\DeclareMathOperator{\Id}{Id}

\DeclareMathOperator{\Ad}{Ad}             
\DeclareMathOperator{\Hom}{Hom\,}           

\DeclareMathOperator{\SL}{SL}

\DeclareMathOperator{\Fix}{Fix}

\DeclareMathOperator{\Sym}{Sym}    

\DeclareMathOperator{\Conj}{Conj}
\DeclareMathOperator{\Span}{Span}
\DeclareMathOperator{\Vect}{\mathbf{Vect}}
\DeclareMathOperator{\As}{As}

\newcommand{\upcurvearrowright}{\rotatebox[origin=c]{180}{$\curvearrowleft$}}
\newcommand{\upcurvearrowleft}{\rotatebox[origin=c]{180}{$\curvearrowright$}}

\makeatletter
\@namedef{subjclassname@2020}{%
  \textup{2020} Mathematics Subject Classification}
\makeatother

\subjclass[2020]{Primary: 57K14. 
Secondary: 14M35, 
57K12, 
18M20. 
}
\keywords{Alexander polynomial, representation variety, Topological Quantum Field Theory, knot quandle.}

\title{Representations of knot groups in $\textrm{AGL}_{1}(\mathbb{C})$ and Alexander invariants}

\author{\'Angel Gonz\'alez-Prieto}
\address{Departamento de \'Algebra, Geometr\'ia y Topolog\'ia, Facultad de Ciencias Matem\'aticas, Universidad Complutense de Madrid, Plaza Ciencias 3, 28040 Madrid Spain, and Instituto de Ciencias Matem\'aticas (CSIC-UAM-UCM-UC3M), C.\ Nicolás Cabrera 13-15, 28049 Madrid Spain.}\email{angelgonzalezprieto@ucm.es}

\author{Javier Mart\'{\i}nez}
\address{Departamento de Matem\'atica Aplicada a la Ingeniería Industrial, Escuela Técnica Superior de Ingeniería y Diseño Industrial, Universidad Politécnica de Madrid, Ronda de Valencia 3, 28012 Madrid Spain.}
\email{javier.martinezmartinez@upm.es}

\author{Vicente Mu\~noz}
\address{Instituto de Matem\'aticas Interdisciplinar and 
Departamento de \'Algebra, Geometr\'{\i}a y Topolog\'{\i}a, Facultad de Ciencias Matem\'aticas, Universidad Complutense de Madrid, Plaza Ciencias 3, 28040 Madrid Spain.}\email{vicente.munoz@ucm.es}

\begin{document}

\begin{abstract}
    This paper reinterprets Alexander-type invariants of knots via representation varieties of knot groups into the group $\textrm{AGL}_{1}(\mathbb{C})$ of affine transformations of the complex line. In particular, we prove that the coordinate ring of the $\textrm{AGL}_{1}(\mathbb{C})$-representation variety is isomorphic to the symmetric algebra of the Alexander module. This yields a natural interpretation of the Alexander polynomial as the singular locus of a coherent sheaf over $\mathbb{C}^*$, whose fibres correspond to quandle representation varieties of the knot quandle. As a by-product, we construct Topological Quantum Field Theories that provide effective computational methods and recover the Burau representations of braids. This theory offers a new geometric perspective on classical Alexander invariants and their functorial quantization.
\end{abstract}

\maketitle

\section{Introduction}\label{sec:introduction}

Since its inception in 1928, as introduced in \cite{Alexanderoriginal}, the Alexander polynomial has been one of the main invariants in knot theory. Over the past century, it has been extensively studied from several perspectives, playing a central role in connecting topological, algebraic and geometric approaches to the classification of knots. Alexander originally derived it from a Dehn presentation of the knot group using a planar projection: a system of linear equations naturally arises, leading to a determinant that defines the invariant.  The Alexander polynomial was long believed to be the only polynomial knot invariant until the discovery of the Jones and HOMFLY-PT polynomials \cite{Homflyoriginal,Jonesoriginal}. Years later, Witten's seminal work \cite{wittenjonespoly} revolutionized knot theory, establishing the framework for Topological Quantum Field Theories (TQFTs) and various categorifications of algebraic knot invariants, notably Khovanov homology \cite{khovanovCategorification} for the Jones polynomial and its variation for the Alexander polynomial \cite{robert2022quantum}. From the perspective of representation theory, a fruitful development has been the construction of Reshetikhin-Turaev invariants \cite{RTinvariants}, which produce invariants of knots using representations of quantum groups. The Kauffman bracket and the Jones polynomial can thus be recovered from representations of certain quantum groups, such as $U_q(\mathfrak{sl}_2)$, leading also to coloured variations and several other insights. 

This paper explores Alexander invariants from the perspective of representation varieties and their quantization. Given a finitely generated group $\Gamma$ and an algebraic group $G$, the \emph{$G$-representation variety} of $\Gamma$ is defined as
$$
\R_G(\Gamma)= \Hom(\Gamma,G).
$$
This set is naturally endowed with an algebraic structure inherited from $G$, justifying the name representation variety. Due to their relationship with representation theory of $\Gamma$, their geometry and algebraic structures is an active area of research.

In this direction, representation varieties, and their associated GIT conjugation action  quotients $\X_G(\Gamma) =\Hom(\Gamma,G)\sslash G$, known as $G$-character varieties, play an important role in knot theory. The first relation appears when we consider the fundamental group of the complement of a knot $K \subset S^3$
$$
    \Gamma_K = \pi_1(S^3 - K),
$$
usually known as the \emph{knot group}. In that case, very subtle information can be extracted from character varieties, such as in the case $G = \sldos$, where the character variety of the knot group gauges the hyperbolic structures on the knot complement \cite{Cooper1994Planecurves,CullerShalen}. 

The aim of this paper is to show that the whole theory of Alexander-type invariants is encoded in the global and local geometric properties of the $\aglC$-representation variety of knots and tangles, being $\aglC$ the group of affine transformations of the complex line. As we will see, this is not merely a reinterpretation of the Alexander invariants; instead, we show that $\aglC$-representation varieties are the natural framework to formulate, study, and propose categorifications for them. In this direction, the first main result of this work relates the (first) Alexander module of a knot $K$ with the $\aglC$-representation variety. Recall that the Alexander module of a knot $K$ is the first homology group of the cyclic cover of the complement $S^3 - K$, which is naturally a $\CC[t,t^{-1}]$-module, where $t$ acts by the deck transformation of the meridian of $K$.

\begin{thmIntro}\label{thmintro:repr-alexander}
    Let $K \subset S^3$ be a knot. Then, the coordinate ring of the representation variety $\R_{\aglC}(\Gamma_K)$ is isomorphic as $\CC[t,t^{-1}]$-algebra to the symmetric algebra over $\CC[t,t^{- 1}]$ of the Alexander module of $K$.
\end{thmIntro}

Theorem \ref{thmintro:repr-alexander} can be seen as a vast generalization of a key prior observation by de Rham \cite{deRham}, 
stating that the Alexander polynomial of $K$ can be computed from the locus of non-abelian representations $\rho: \Gamma_K \to \aglC$. Explicitly, he showed that if we prescribe the image of a meridian $\gamma \in \Gamma_M$ to be of the form $\rho(\gamma) =\left(\begin{smallmatrix}
    1 & a \\ 0 & t
\end{smallmatrix}\right)$, with $a \in \CC$ and $t \in \CC^* = \CC - \{0\}$, then a non-abelian representation of $\Gamma_K$ into $\aglC$ exists if and only if $t$ is a root of the Alexander polynomial of $K$ (see \cite{Burde}). The theorem has been a starting point for significant results concerning the local structure of the $\sldos$-character variety \cite{heusenerportipeiro}, demonstrating that non-abelian reducible representations can be deformed into irreducible ones via cohomological methods. This approach has proven useful in other contexts \cite{heusenerportipsl,Kozaideformations}, and has been extended to other groups \cite{Burde-deRham2} as well as to twisted Alexander invariants \cite{taehee,silverwilliams}.

In this direction, in this work we revisit these ideas from a geometric perspective, focusing on the $\aglC$-representation variety of a knot group. Given an oriented knot $K \subset S^3$, we focus on the regular morphism between representation varieties
\begin{equation}\label{eq:Alexander-fibration-intro}
\pi_K: \R_{\agl}(\Gamma_K) \longrightarrow \R_{\CC^{\ast}}(\Gamma_K) = \CC^*
\end{equation}
induced by the projection onto the diagonal entry $\aglC \to \CC^{\ast}$, which we will refer to as the \emph{Alexander fibration}. 
We shall show that the fibres of the Alexander fibration (\ref{eq:Alexander-fibration-intro}) can also be understood in purely representation-theoretic terms. Indeed, for $t \in \R_{\CC^{\ast}}(\Gamma_K) = \CC^*$, the fibre $\pi^{-1}_K(t)$ is naturally connected to the quandle theory of $K$, and explicitly it can be identified with the quandle representation variety
$$
\R_{\AQ{t}}(Q_K) = \Hom_\Qdl(Q_K, \AQ{t})
$$
of Joyce's knot quandle $Q_K$ introduced in \cite{Joycequandle}, into the $t$-Alexander quandle $\AQ{t}$. These quandle representation varieties turn out to be complex vector spaces in a natural way, exhibiting the Alexander fibration as a `singular' complex vector bundle (i.e.,\ as the total space of a coherent sheaf). This provides a reinterpretation of many Alexander invariants in terms of classical algebraic invariants of the character variety. In particular, we show that the exceptional locus of the Alexander fibration is determined by a stratification of $\R_{\CC^*}(\Gamma_K) = \CC^*$, known as the Alexander stratification, whose vanishing ideals coincide with the elementary ideals of the Alexander module. In particular, this Alexander fibration encodes the Alexander polynomial as anticipated by de Rham's result. Furthermore, the Alexander module re-arises precisely as the set of global sections of the Alexander fibration.

Using these ideas, we will develop a Topological Quantum Field Theory explicitly computing the Alexander module of a knot. For this purpose, let $\cT$ be the \emph{category of tangles} on a cylinder, whose objects are oriented points on a plane with morphisms being oriented tangles between them. Observe that knots and links appear in $\cT$ as endomorphisms of the empty set of points. Then, the main result of this paper is the following theorem.

\begin{thmIntro}\label{thm-intro:TQFT-global}
There exists a braided monoidal functor
$$
\cA: \cT \longrightarrow \Span(\SVB),
$$
computing the Alexander module of knots.
\end{thmIntro}

In Theorem \ref{thm-intro:TQFT-global}, the target category $\Span(\SVB)$ is the so-called span category of singular vector bundles. Explicitly, the functor $\cA$ associates to any tangle $T$ the span of $\aglC$-representation varieties
$$
\xymatrix{
& \R_{\aglC}(\Gamma_T) \ar[dl]_{i} \ar[dr]^{j} & \\
\R_{\aglC}(\Gamma_\epsilon) & & \R_{\aglC}(\Gamma_\varepsilon)
}
$$
where $\epsilon$ and $\varepsilon$ are the boundary points $\partial T=\overline{\epsilon} \sqcup \varepsilon$. There, the maps $i: \R_{\aglC}(\Gamma_T) \to \R_{\aglC}(\Gamma_\epsilon)$ and $j: \R_{\aglC}(\Gamma_T) \to \R_{\aglC}(\Gamma_\varepsilon)$ are the restriction maps to the boundaries. Each of these representation varieties are on themselves a singular vector bundle for their corresponding Alexander fibrations, and the restriction maps $i$ and $j$ preserve these fibered structures. In particular, if $K$ is a knot, seen as a tangle between empty sets, the corresponding span will be
$$
\xymatrix{
0 & \R_{\aglC}(\Gamma_K) \ar[l] \ar[r] & 0, \\
}
$$
in such a way that the global sections of this singular vector bundle is precisely the symmetric algebra of the Alexander module of $K$.

The functor $\cA: \cT \to \Span(\SVB)$ must be understood as a Topological Quantum Field Theory (TQFT for short) from the category $\cT$ of tangles with values in the category $\Span(\SVB)$, in the sense of Atiyah \cite{atiyah1988topological} and Segal \cite{segal2001topological}. These TQFTs are the analogues of a functorial quantization of a given field theory, in which manifolds are quantized into a (Hilbert) vector space and bordisms are translated into a `path integral' operation between the boundary Hilbert spaces. In the context of knot theory, it is customary to  replace the target category of vector spaces by another monoidal category ($\Span(\SVB)$ in our case) better adapted to the geometric problem but mimicking all the properties of a quantization procedure. These TQFT-based methods have been extensively studied in the literature due to their ability to capture and provide effective methods of computation of subtle invariants, such as (coloured) Jones or HOMFLY-PT polynomials, see \cite{stroppel2022categorification} for a survey on the state-of-art. 

The TQFT constructed in Theorem \ref{thm-intro:TQFT-global} can also be understood fibrewise, using the quandle representation interpretation. Therefore, if $\Span(\Vect_{\CC})$ is the category of spans of vector spaces, we get the following result.

\begin{thmIntro}\label{thm-intro:TQFT-fibre}
For any $t\in \CC^{\ast}$, there exist a braided monoidal functor $\cA_t$,
$$
\cA_t: \cT\longrightarrow\Span(\Vect_{\CC}),
$$
which computes the quandle $\AQ{t}$-representation variety.
\end{thmIntro}

In this manner, the functor $\cA_t$ associates, to any tangle $T$ with $\partial T=\overline{\epsilon} \sqcup \varepsilon $, the span of vector spaces
$$
\xymatrix{
\R_{\AQ{t}}(Q_\epsilon) & \R_{\AQ{t}}(Q_T) \ar[l]_{i} \ar[r]^{j} & \R_{\AQ{t}}(Q_\varepsilon),
}
$$
corresponding to the fibres over $t \in \CC^*$ of the span $\cA(T)$. In this setting, the restriction maps $i: \R_{\AQ{t}}(Q_T) \to \R_{\AQ{t}}(Q_\epsilon)$ and $j: \R_{\AQ{t}}(Q_T) \to \R_{\AQ{t}}(Q_\varepsilon)$ are just linear maps between complex vector spaces. Another advantage is that the TQFT of Theorem \ref{thm-intro:TQFT-fibre} is highly computable, and can be completely described for the generators of $\cT$ in terms of the quandle structure of the Alexander quandle $\AQ{t}$, as done in Section \ref{sec:explicitcalculations}.

However, perhaps the most striking consequence of Theorem \ref{thm-intro:TQFT-fibre} is that it automatically recovers the celebrated Burau representations of braids. To this purpose, consider the subcategory $\cB$ of $\cT$ of braids, i.e.\ tangles that are only allowed to move forward. The endomorphisms in $\cB$ of the set of $n$ points thus coincide with the Artin braid group $\textup{B}_n$. It turns out that $\cA_t$ restricted to $\cB$ gives rise to a functor
$$
   	\cA_t|_{\cB}: \cB \longrightarrow \Vect_\CC
$$
onto the genuine category of vector spaces, leading to a TQFT on $\cB$ in the classical sense. In this manner, given a braid $B \in \textup{B}_n$, we get a linear map
$$
    \cA_t(B): \R_{\AQ{t}}(\Gamma_n) = \CC^n \to \R_{\AQ{t}}(\Gamma_n) = \CC^n.
$$
We prove in Section \ref{sec:braids} that this map coincides with the (unreduced) Burau representation of $B$, as defined in \cite{burau1935zopfgruppen,KasselTuraevbraidgroups}. This provides a novel interpretation of Burau representations purely in terms of representation varieties.

This is particularly interesting when relating knots and braids. Recall that Alexander's theorem states that every knot can be represented as the closure $\overline{B}$ of a braid $B$. This fact provides a way to compute the Alexander polynomial using the Burau representation, as shown, for instance, in \cite{KasselTuraevbraidgroups}. In this direction, the TQFT that we define $\mathcal{A}_t$, 
furnishes a new direct way to prove this relation, since the braid closure operator corresponds to taking the fix locus of the Burau representation, recovering in a straightforward way the standard formulas for the Alexander polynomial.

We finish this section by pointing out that the classical results re-proven in this work are just a few examples of the flexibility and simplicity that the developed TQFT-based framework confers. Thanks to the functor $\mathcal{A}_t: \cT\to \Span(\Vect_{\CC})$, we can provide alternative proofs of well-known results in knot theory involving Alexander polynomials with a direct approach. Furthermore, we expect to exploit these techniques to produce new explicit descriptions of the Alexander module for certain families of knots, which will be presented in a separate note \cite{rationalknots}.

\subsection*{Related work} The Alexander theory of a knot is a widely studied topic that has been analyzed through different viewpoints. It was soon recognized that the Alexander polynomial is an invariant associated with the first homology group of the infinite cyclic cover of the knot complement. Consequently, it can be obtained from the Seifert matrix or via Fox calculus, as noted in reference texts on knot theory \cite{BurdeZieschangbook,Kauffmanonknotsbook,Lickorish}. Conway later introduced the potential function \cite{Conwayenumeration} --now known as the Alexander-Conway polynomial-- and demonstrated how to compute it recursively using \emph{skein relations}, emphasizing the importance of local relations that were already present in Alexander's original work. Subsequence advances in knot theory have further explored Alexander invariants from other points of view, such as the Vassiliev invariants \cite{bar1995vassiliev,vassiliev1992cohomology}, the Reidemeister torsion \cite{Milnortorsion,Turaevtorsion} or knot Floer homology \cite{knotfloer1}, among many others. 

The relation of the Alexander polynomial with representation theory has also been explored in the literature. Back in the 60's, de Rham \cite{deRham} already reported that abelian representations of the $\sldos$-character variety appear precisely over the roots of the Alexander polynomial. This character variety actually plays a central role in the construction of other celebrated knot invariats, such as the A-polynomial \cite{Cooper1994Planecurves} and the brand-new slope invariant \cite{benard2021slope}. This relation has been clarified through the theory of skein modules \cite{SikoraPrz}, by means of its relation with quantum groups and quantum field theories. Using the representation theory of surface groups, Turaev \cite{Turaevskeinquantization} showed that the Kauffman bracket skein algebra provides a quantization of the $\sldos$-character variety. Moreover, representations of this algebra generalize the Kauffman bracket invariant for links \cite{bonahonwong2, bonahonwong1}. For 3-manifolds, it has recently been proved that the Kauffman skein module is finite dimensional \cite{finiteskein} and has been computed in certain cases \cite{sikora3manifolds}.

However, to the best of our knowledge, it has not been reported in the literature such a clear relation between the geometry of the $\aglC$-representation variety and Alexander modules as the one provided by our Theorem \ref{thmintro:repr-alexander}. Perhaps the closest work in this direction was conducted by Hironaka \cite{hironaka1997alexander}, who exploited the idea of attaching a coherent sheaf from a purely cohomological point of view with the aim of identifying non-K\"ahler finitely generated groups. In that work, equations defining the jumping loci for the first cohomology group of one-dimensional representations of a finitely presented group were computed via Fox calculus. These variations in dimension naturally induce a stratification of the character variety of a knot group, recovering the Alexander stratification in terms of coherent sheaves. However, to the best of our understanding, this is a purely ad-hoc construction with no relation with the $\aglC$-representation variety.

It is worth noticing that there exists in the literature several approaches to construct functorial quantum field theories for Alexander-type invariants. Apart from the groundbreaking work by Witten quantizing the Jones polynomial by means of Chern-Simons theory, perhaps the first construction of a TQFT for Alexander invariants was developed by Frohman and Nicas \cite{frohman1992alexander}. In this work, the authors constructed a functor $Z: \mathcal{C} \to \mathbb{P}\textup{\textbf{Ab}}$ from a suitable subcategory $\mathcal{C}$ of the category of $3$-dimensional bordisms to the category $ \mathbb{P}\textup{\textbf{Ab}}$ of ``projective'' abelian groups, i.e.\ the usual category of abelian groups but with morphisms defined only up to re-scaling. This functor assigns to a connected surface $\Sigma$ the homology $H_\star(\R_{\textup{U}(1)}(\pi_1(\Sigma)))$ of the associated $\textup{U}(1)$-representation variety, and is designed in such a way that it quantizes the intersection form. In this manner, this TQFT captures the Seifert matrix of a Seifert surface for a knot and thus the construction recovers the Alexander polynomial of the knot from the vacuum state associated to the bordism resulting from applying longitudinal surgery to the knot. It is worth mentioning that the use of $\textup{U}(1)$-representations is very natural in this context since $R_{\textup{U}(1)}(\pi_1(\Sigma)) = H^1(\Sigma)$ by the Hurewicz theorem.

The construction of Frohman-Nicas has been subsequently extended to more general settings. In the work \cite{bigelow2015alexander} by Bigelow, Cattabriga and Florens, the TQFT was extended to the category of tangles by quantizing the Burau representation of braids (see also \cite{long1989linear,long1989blinear}). Complementarely, Florens and Massuyeau extended in \cite{florens2016functorial} the Frohman-Nicas TQFT to the category of bordisms $W$ equipped with a representation $\pi_1(W) \to A$, for $A$ a fixed abelian group. Both constructions are essentially cohomological, being a a key tool the so-called Lescop's Alexander function \cite{lescop1998sum}, which operates by considering exterior products of the Alexander module of a knot.

The TQFT developed by Frohman-Nicas and its extensions are inherently related to the ones constructed in Theorems \ref{thm-intro:TQFT-global} and \ref{thm-intro:TQFT-fibre}. Both constructions exploit the representation theory of the manifolds involved to quantize the Alexander polynomial, making use of the restriction maps onto the boundary. However, they are also very different in essence for several reasons. First, the Frohman-Nicas TQFT only uses abelian representations, whereas the non-abelianity of the group $\aglC$ plays a fundamental role in our construction to capture the additive version of the Wirtinger relations. But, perhaps more importantly, the Frohman-Nicas construction is purely homological in nature, exploiting the homology theory of the involved manifolds, and thus linking with the Alexander polynomial through the intersection form. In sharp contrast, our TQFT is built upon subtle geometric properties of the $\aglC$-representation theory, capturing the Alexander polynomial as the singular locus of a singular vector bundle. This also implies that our TQFT can be effectively computed from purely combinatorial information, which can be used to perform explicit calculations in highly involved knots.

It is an interesting prospective work to develop a general framework able to encompass both constructions, for instance, based on pull-push quantization on the derived category of coherent sheaves, à la Fourier-Mukai, as done for instance in \cite{gonzalez2020lax,vogel2024cohomology}. This potential relation can also be envisaged through the work of Cimasoni and Turaev \cite{cimasoni2005lagrangian, cimasoni2006lagrangian}, where the authors use this homological machinery to construct another TQFT from the category of tangles into a category of ``Lagrangian relations'', which might be seen as the homological analogue of our category of spans of singular vector bundles. In fact, the construction has been recently categorified in \cite{cimasoniconway}, where the authors construct a 2-functor from the category of tangles to the category of Lagrangian cospans. In contrast to the representation-theoretic framework explored in this manuscript, these approaches are homological in nature and do not involve representation varieties; computations have so far been carried out primarily under certain restrictions, such as topologically trivial tangles.

Finally, we want to emphasize that the quantization based on functorial field theory, as the one developed in the aforementioned works, differ from the `categorification' approaches, such as Khovanov-like methods \cite{khovanovCategorification,robert2022quantum}, which capture polynomial invariants in a distinct way, namely, by constructing a homology theory whose (graded) Euler characteristic recovers the polynomial.

\subsection*{Structure of the manuscript} Section \ref{sec:2} analyzes the $\aglC$-representation variety of a knot group and defines the Alexander fibration, which plays a central role throughout the paper. Section \ref{sec:Alexmodulematrix} revisits the Alexander module from the point of view of Fox calculus. In particular, Section \ref{sec:Alexander-module-top} explores its connection with $\aglC$-representations, while its interpretation as a coherent sheaf is presented in Section \ref{sec:alexsingularvectorbundle}. 

Section \ref{sec:quandles} introduces fundamental concepts of quandle theory and examines its relationship with the fibre of the Alexander fibration. The construction of a TQFT for the Alexander module is the focus of Section \ref{sec:tanglesandalexquandles}. After reviewing the tangle category and the $\Span(\Vect_{\CC})$ categories in Sections \ref{sec:tangle-cat} and \ref{sec:span-cat}, respectively, the functor $\cA_t$ and its main properties are studied in Section \ref{sec:TQFT}. The functor is then extended to define a global TQFT of singular vector bundles over $\CC^*$ in Section \ref{sec:TQFTsingular}. 
An explicit description of the functor in terms of the generators of the tangle category is the subject of Section \ref{sec:explicitcalculations}. The manuscript concludes by analyzing the braid subcategory and the Burau representation as particular cases in Section \ref{sec:braids}.

\subsection*{Acknowledgements}

The first-named author was partially supported by the grant ``Computational, dynamical and geometrical complexity in fluid dynamics'', Ayudas Fundación BBVA a Proyectos de Investigación Científica 2021, by Bilateral AEI-DFG project: Celestial Mechanics, Hydrodynamics, and Turing Machines, and by Ministerio de Ciencia e Innovaci\'on Grant PID2021-124440NB-I00 (Spain).
The third-named author was partially supported by Project MINECO (Spain)  PID2020-118452GB-I00.  

\section{$\agl$-representations of knots} \label{sec:2}
Let $K\subset S^3$ be a knot and denote by $K^c=S^3-K$ its complement. 
One of the most important invariants associated to $K$ is its \emph{knot group}, which is the fundamental group $\G_K:=\pi_1(K^c)$. The mere information of this group is not a complete invariant (e.g.\ the square knot and the granny knot have isomorphic knot groups), but $\G_K$ together with the so-called peripheral subgroup is a complete invariant \cite{waldhausen1968irreducible}.

There exists a simple procedure to present $\G_K$ from a regular projection of $K$, the so-called Wirtinger presentation. It expresses $\G_K$ as the finitely generated group,
\begin{equation} \label{eqn:pi1presentation}
\G_K=\langle x_1,\ldots,x_n \,|\, r_1,\ldots,r_{n} \rangle,
\end{equation}
where $n$ is the number of crossings, and each $x_i$ geometrically represents a loop around an arc between crossings. Besides, each relation $r_k$ is of the form 
\begin{equation} \label{eqn:crossingequation}
\textup{(I)} \quad x_{i+1}=x_j x_ix_j^{- 1}, \qquad \textup{or} \qquad
\textup{(II)}\quad x_{i+1}=x_j^{- 1}x_ix_j\, ,
\end{equation} 
depending on the type of the crossing, which is positive or negative when the knot is oriented. 
\begin{figure}[h]
\includegraphics[width=7cm]{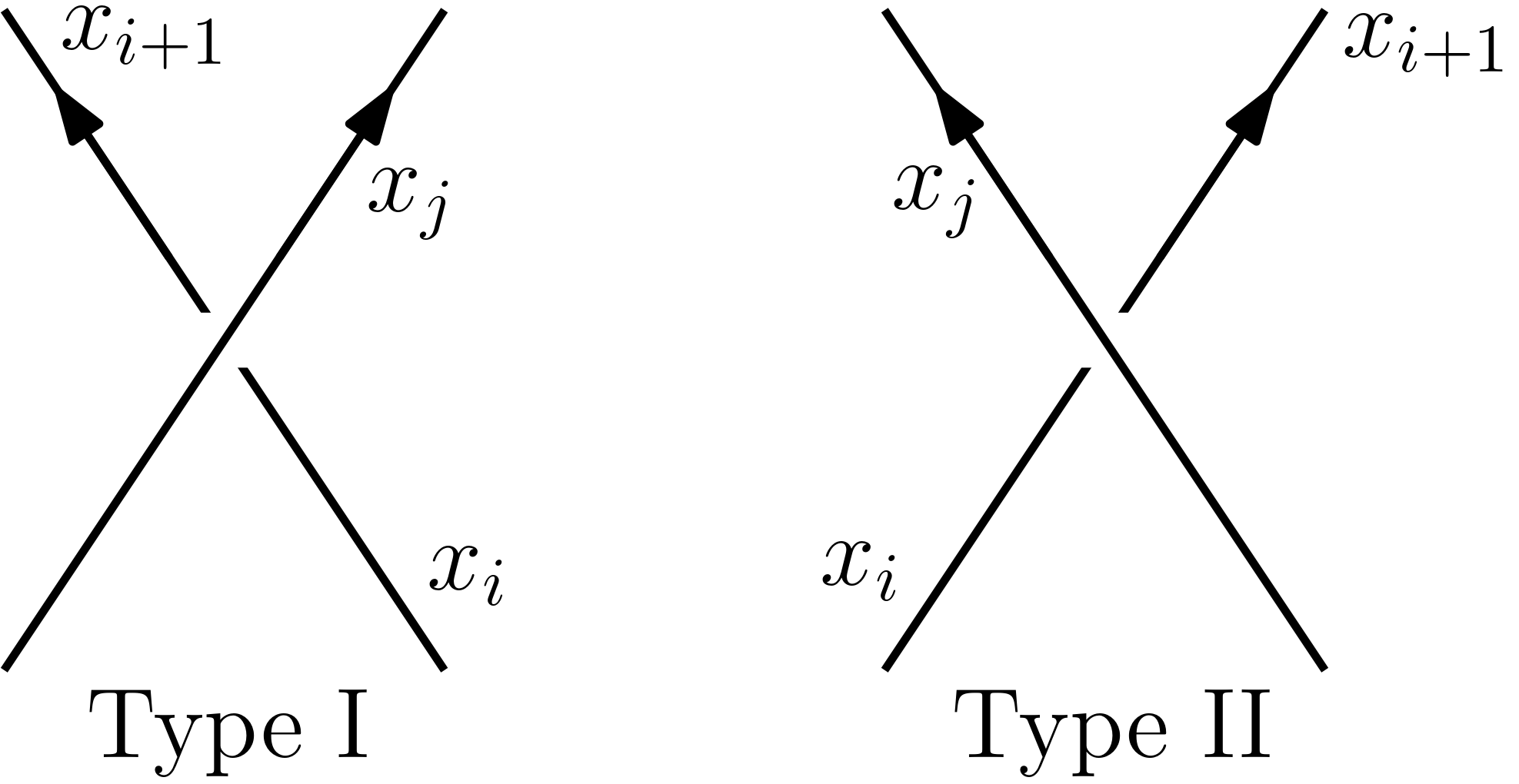}
\caption{Wirtinger generators for positive (type I) and negative (type II) crossings. \label{fig:wirtinger}}
\end{figure}
We note that in (\ref{eqn:pi1presentation}), one relation is redundant, since $r_1\cdots r_n=1$. Therefore we have
\begin{equation} \label{eqn:pi1presentation2}
\G_K=\langle x_1,\ldots,x_n \,|\, r_1,\ldots,r_{n-1} \rangle.
\end{equation}

Given a finitely presented group $\Gamma=\langle x_1,\ldots, x_n \,|\, r_1,\ldots, r_s \rangle$ and any complex affine algebraic group $G$, a \emph{representation} of $\Gamma$ into $G$ is a group homomorphism $\rho: \Gamma \rightarrow G$, which is determined by the image of the generators $A_i :=\rho(x_i) \in G$ for $1 \leq i \leq n$, subject to the relations $r_j(A_1,\ldots, A_n)$, for $1\leq j\leq s$. As a consequence, the \emph{representation variety} is the affine algebraic set
$$
\R_G(\Gamma)= \Hom(\Gamma,G)=\lbrace (A_1,\ldots,A_n)\in G^n \,|\, r_j (A_1,\ldots, A_n)=\id, 1\leq j\leq s\rbrace \subset G^n\, .
$$

Assuming that $G$ is a linear group, i.e.\ $G\subset \GL{r}(\CC)$ for some $r > 0$, we use the following classical terminology. We say that $\rho$ is
\emph{reducible} if there is some proper subspace $V \subsetneq \CC^r$ such that $\rho(\gamma)(V)\subset V$ for all $\gamma \in \Gamma$, and \emph{irreducible} otherwise.
We say that $\rho$ is \emph{abelian} if the image $\rho(\Gamma)$ is an abelian subgroup of $\GL{r}(\CC)$, and \emph{non-abelian} otherwise. 
Any abelian representation with $r\geq 2$ is automatically reducible.

In this work, we shall focus on the group $G = \aglC$ of affine transformations of the complex line. It can be described as a matrix group as
\begin{equation}\label{eqn:agldescription}
    \aglC = \left\{\left.\begin{pmatrix}
    1 & 0 \\ a & t
    \end{pmatrix} \,\right|\, a \in \CC, t \in \CC^*  \right\}.
\end{equation}
For simplicity, we will abbreviate $\agl=\aglC$. Notice that algebraically, $\agl = \CC^* \ltimes \CC$, for the action of $\CC^*$ on $\CC$ by $t \cdot a = ta$.

Of particular interest for this work will be the study of the $\agl$-representation variety of a knot group
(\ref{eqn:pi1presentation2}), that is, the space
$$
\R_{\agl}(K) := \R_{\agl}(\G_K)= \lbrace (A_1,\ldots,A_n) \in (\agl)^n \,|\, r_j(A_1,\ldots,A_n)=\id, 1\leq j\leq n-1 \rbrace,
$$
where $r_j$ are the Wirtinger relations (\ref{eqn:crossingequation}). Since the generators of the knot group represent meridians and the Wirtinger relations are described by conjugacy relations, it is straightforward that the traces of the images of the generators satisfy $\tr(A_i)=\tr(A_j)$ for all $1 \leq i,j \leq n$. As a consequence, if we write
\begin{equation}\label{eqn:matricesinagl}
A_i=\begin{pmatrix} 1 & 0 \\ a_i & t_i \end{pmatrix},\qquad a_i \in \CC \textrm{ and } t_i \in \CC^* ,
\end{equation}
we get that $t:=t_i$ is constant for all $i$ and therefore $ \R_{\agl}(K) \subset \CC^* \times \CC^{n}$, the latter space parametrized by $(t, a_1,\ldots,a_n)$. Besides, each relation of type (I), $A_jA_iA_j^{-1}=A_{i+1}$, is equivalent to the equation $a_j+ta_i-ta_j=a_{i+1}$, which we rewrite as
\begin{equation}\label{eqn:typeIlineareq}
ta_i+(1-t)a_j-a_{i+1}=0.
\end{equation}
Analogously, each relation of type (II), $A_j^{-1}A_iA_j=A_{i+1}$, yields an analogous equation $-a_jt^{-1}+a_j+t^{-1}a_i=a_{i+1}$, equivalent to
\begin{equation} \label{eqn:typeIIlineareq}
t^{-1}a_i+(1-t^{-1})a_j-a_{i+1}=0.
\end{equation}
This provides an explicit description of $\R_{\agl}(K)$ from a planar projection of $K$.

We note that equations \eqref{eqn:typeIlineareq} and \eqref{eqn:typeIIlineareq} are linear in $\CC[t,t^{-1}]$. As a consequence, we may consider the following fibration that plays a central role in this work. Notice that the diagonal entry subgroup $\CC^* \subset \agl$ defines a restriction regular morphism $\pi_K: \R_{\agl}(K) \to \R_{\CC^{\ast}}(K)$, but this latter space is naturally isomorphic to $\CC^*$ since $\R_{\CC^{\ast}}(K) = \Hom(\Gamma_K, \CC^\ast) = \Hom(H_1(K^c), \CC^\ast) = \CC^\ast$, as $H_1(K^c) = \ZZ$ is generated by the meridian.

\begin{defn}
The regular morphism
\begin{equation} \label{eqn:representationfibration}
\pi_K: \R_{\agl}(K) \longrightarrow \CC^{\ast}
\end{equation}
is called the \emph{Alexander fibration}.
\end{defn}

Notice that the fibre $\mathcal{S}_t := \pi_K^{-1}(t)$ of the Alexander fibration over $t \in \CC^*$ is a complex vector space given by the set of solutions $(a_1,\ldots, a_n) \in \CC^n$ to the linear system of equations described in \eqref{eqn:typeIlineareq} and \eqref{eqn:typeIIlineareq}, one for each crossing.

In this context, abelian $\agl$-representations of a knot can be fully identified since they are parametrized by $\CC^{\ast}\times \CC$. Indeed, any abelian representation factors through the abelianization $H_1(K^c)\cong \ZZ$, so it is defined by the image of any meridian, $\rho(x_i)$, which is determined by a single element $(t,a)\in \CC^{\ast}\times \CC = \agl$.

Regarding non-abelian affine representations of $\G_K$, the following fundamental theorem originally proven by de Rham and Burde in \cite[Section 6.1]{Cooper1994Planecurves} plays a fundamental role. We restate it in terms of the fibration given in \eqref{eqn:representationfibration}, and we will re-prove it in Section \ref{sec:Alexander-module-top}.

\begin{thm}[Burde--de Rham]\label{thm:2.2}
 There exists a non-abelian representation $\rho: \G_K \to \agl$ such that $\pi_K(\rho)=t$ if and only if $\Delta_K(t)=0$, where $\Delta_K(t)$ denotes the Alexander polynomial of the knot $K$.   
\end{thm}

In other words, the exceptional locus of the Alexander fibration $\pi_K:\R_{\agl}(K)\rightarrow \CC^{\ast}$ is precisely the algebraic variety defined by the Alexander polynomial, $V(\Delta_K)\subset \CC^{\ast}$, which is a collection of points. For any other $t\in \CC^{\ast}-V(\Delta_K)$, the only representations are abelian and parametrized by $a\in \CC$.

\section{The Alexander module and the Alexander matrix} \label{sec:Alexmodulematrix}

In order to provide a clearer picture of the fibration \eqref{eqn:representationfibration}, 
we recall here some elementary facts regarding the Alexander polynomial construction from the Alexander module and its connection with Fox calculus. 
We refer to \cite{BurdeZieschangbook} for further details. 

\subsection{Fox calculus for the Alexander module}\label{sec:Alexander-module-top}

Throughout this section, let us fix a knot $K \subset S^3$ with complement $K^c = S^3 - K$. The commutator subgroup $[\pi_1(K^c),\pi_1(K^c)] \triangleleft \pi_1(K^c)$ defines an infinite cyclic covering $K^c_{\infty} \rightarrow K^c$, with deck group isomorphic to $H_1(K^c)=\pi_1(K^c)/[\pi_1(K^c),\pi_1(K^c)]\cong \mathbb{Z}=\langle t \rangle$, which is usually called the \emph{infinite cyclic covering} of $K^c$. The cyclic action of $t$ endows each homology group $H_k(K^c_{\infty})$ with a $\mathbb{Z}[t,t^{-1}]$-module structure, giving rise to a collection of modules that are invariants of the knot, the \emph{Alexander modules} of $K$.

The (first) \emph{Alexander module} is the finitely presented $\ZZ[t,t^{-1}]$-module $H_1(K^c_{\infty})$ and the \emph{Alexander matrix} $A_K(t)$ is the $(n-1)\times n$-matrix, with coefficients in $\ZZ[t,t^{-1}]$, associated to any presentation of the module, where two Alexander matrices are declared equivalent if they present isomorphic modules.

Associated to the Alexander matrix $A_K(t)$, we may consider its \emph{elementary ideals} $E_K^k \subset \ZZ[t,t^{-1}]$, $k\in \ZZ$, defined as
$$
E_K^k =\begin{cases} 0, & \text{if } k \leq 0,  \\ \text{ideal generated by the $(n-k)\times(n-k)$-minors of $A_K$}, & \text{if } 0<k\leq n-1, \\
\ZZ[t,t^{-1}], & \text{if } k\geq n. \end{cases}
$$
Elementary ideals are invariants of $A(t)$, hence invariants of the knot, and they form an ascending chain
\begin{equation} \label{eqn:elementaryideals}
0=E_K^0\subset E_K^1 \subset \ldots \subset E_K^{n} =\ZZ[t,t^{-1}].
\end{equation}

The $k$\emph{-th Alexander polynomial} $\Delta_K^k(t)$ of the knot $K$ is defined as the greatest common divisor of the elements of the ideal 
$E_K^k$, defined up to units $\pm t^k$, since $\ZZ[t,t^{-1}]$ is a unique factorization ring. In other words, $\Delta_K^k(t)$ is the generator of the smallest principal ideal containing $E_K^k$. Clearly, the first elementary ideal $E_K^1$ is a principal ideal, 
generated by a single polynomial $\Delta_K^1(t)$, also denoted by $\Delta_K(t)$, the \emph{Alexander polynomial}, that appeared in Theorem \ref{thm:2.2}.

Fox Calculus provides a systematic way to obtain an Alexander matrix $A(t)$ for any knot $K$. If $\Free{n}=\langle x_1,\ldots, x_n \rangle$ denotes the free group on $n$ generators, the \emph{Fox derivative} is the $\ZZ$-module map $\frac{\partial}{\partial x_i}: \ZZ[\Free{n}] \rightarrow \ZZ[\Free{n}]$ defined by the following properties:
\begin{enumerate}
    \item $\frac{\partial 1}{\partial x_i}=0$, 
    \item $\frac{\partial x_j}{\partial x_i}=\delta_{ij}$, 
    \item $\frac{\partial w_1w_2}{\partial x_i}=\frac{\partial w_1}{\partial x_i}
    +w_1\frac{\partial w_2}{\partial x_i}$, for words $w_1,w_2 \in F_n$, 
\end{enumerate}
and we extend linearly to $\ZZ[\Free{n}]$.
In particular, $\frac{\partial wx_j}{\partial x_i}=\frac{\partial w}{\partial x_i}+w\delta_{ij}$
and $\frac{\partial x_jw}{\partial x_i}=\delta_{ij}+x_j\frac{\partial w}{\partial x_i}$, for any word $w\in \Free{n}$ and any generators $x_i,x_j$.
Also one can compute $\frac{\partial x_i^{-1}}{\partial x_i}=-x_i^{-1}$.

Now, given a finitely generated group $\Gamma=\langle x_1,\ldots, x_n \,|\, r_1,\ldots, r_s \rangle$, we consider the $\ZZ$-module maps $p: \ZZ[\Free{n}] \rightarrow \ZZ[\Gamma]$ and $H: \ZZ[\Gamma] \rightarrow \ZZ[\Gamma^{ab}]$, induced by the projection map $p:F_n\rightarrow \Gamma$ and the abelianization map $H: \Gamma \rightarrow \Gamma^{ab}$, where $\Gamma^{ab}$ denotes the abelianization of $\Gamma$. 
Denoting $\{w\}=(H\circ p)(w)$, the Alexander matrix of $\Gamma$ is the matrix with coefficients in $\ZZ[\Gamma^{ab}]=\ZZ[t,t^{-1}]$ given by
$$
A_K=(a_{ij})= \left(  \left\{ \frac{\partial r_i}{\partial x_j} \right\} \right).
$$

When $\G=\G_K$ is a knot group (\ref{eqn:pi1presentation2}), we obtain a $(n-1)\times n$-matrix with coefficients in $\ZZ[t,t^{-1}]$, that is precisely an Alexander matrix for the Alexander module $H_1(K^c_{\infty})$. This is due to the fact that Fox derivatives represent the boundary maps for a suitable chain complex for $K^c_{\infty}$, 
and taking into account the $\ZZ[t,t^{-1}]$-module structure on $H_1(K^c_{\infty})$. 

Now, consider a knot group with Wirtinger presentation $\G_K=\pi_1(K^c)=
\langle x_1,\ldots, x_n \,|\, r_1,\ldots, r_{n-1} \rangle$ where relations are given in \eqref{eqn:crossingequation}. 
The map $(H\circ p(x_i))=\{x_i\}=t$.
Let us consider a Wirtinger relator $w=x_jx_ix_j^{-1}x_{i+1}^{-1}$. Its Fox derivatives are given by
\begin{equation} \label{eqn:FoxdertypeI}
 \left\{ \frac{\partial w}{\partial x_i}\right\}=t, \quad 
 \left\{\frac{\partial w}{\partial x_j}\right\}=1-t, \quad 
 \left\{\frac{\partial w}{\partial x_{i+1}}\right\}=-1. 
\end{equation}
Analogously, if we take a relator $w=x_j^{-1}x_ix_jx_{i+1}^{-1}$, we obtain that
\begin{equation} \label{eqn:FoxdertypeII}
 \left\{\frac{\partial w}{\partial x_i}\right\}=t^{-1}, \quad 
 \left\{\frac{\partial w}{\partial x_j}\right\}=1-t^{-1}, \quad 
 \left\{\frac{\partial w}{\partial x_{i+1}}\right\}=-1. 
\end{equation}

\cite[Chapter 11]{Kauffmanonknotsbook} that $H_1(K^c_{\infty})$ is the $\ZZ[t,t^{-1}]$-module generated by $[x_i], i=1,\ldots, n$, under the abelianized Wirtinger relations:
\begin{align}
  t[x_i]+(1-t)[x_j]-[x_{i+1}]  & =0, \label{eqn:XinftytypeIrel} \\
  t^{-1}[x_i]+(1-t^{-1})[x_j]-[x_{i+1}] & = 0, \label{eqn:XinftytypeIIrel}
\end{align}
according to the Wirtinger presentation of the knot. As it is also stated in \cite[section 11.2]{Kauffmanonknotsbook}, the linear relations 
\eqref{eqn:XinftytypeIrel}, \eqref{eqn:XinftytypeIIrel}  in $\ZZ[t,t^{-1}]$ 
of the Alexander module are precisely the ones arising in any affine representation, from \eqref{eqn:typeIlineareq} and \eqref{eqn:typeIIlineareq}. 

\begin{cor}\label{cor:Writinger-presentation-fox}
Given a knot group $\G_K=\pi_1(K^c)$ with a Wirtinger presentation (\ref{eqn:pi1presentation2}),
and a representation $\rho: \G_K \rightarrow \agl$, 
the Fox derivatives of a word $w\in \G_K$ are determined by the off-diagonal entry $b$ of $\rho(w)\in \agl$.
\end{cor}

\begin{ex}
Let $\Gamma_K$ be a knot group and consider a word in its first two generators, 
$w=x_1^2x_2x_1^{-1}$. Applying the properties that define the Fox derivative, it is straightforward to verify that 
$$
 \left\{ \frac{\partial w}{\partial x_1} \right\} =1+t-t^2, \qquad 
 \left\{ \frac{\partial w}{\partial x_2} \right\} = t^2.
$$
Now, let us take $\rho:\G_K \rightarrow \agl$, where $\rho(x_i)=\left(\begin{smallmatrix} 1 & 0 \\ b_i & t \end{smallmatrix} \right)$. We get
$$
\rho(w)=\begin{pmatrix}
1 & 0 \\ (1+ t- t^2)b_1 + t^2b_2 & t^2
\end{pmatrix},
$$
from which both Fox derivatives can be extracted.
\end{ex}

\subsection{$\agl$-representation varieties and the Alexander module}

Given a complex algebraic variety $X$, we denote its sheaf of regular algebraic functions by $\mathcal{O}_X$. In particular, its coordinate ring of global regular functions is $\CC[X] = \mathcal{O}_{X}(X)$.
If we consider $X = \R_{\agl}(K)$, notice that the morphism defining the Alexander fibration 
$$
\pi_K: \R_{\agl}(K) = \Spec \CC[\R_{\agl}(K)] \to \R_{\CC^*}(K) = \Spec \CC[t, t^{-1}]
$$ 
induces a $\CC$-algebra morphism
$$
    \pi_K^*: \CC[t, t^{-1}] \to \CC[\R_{\agl}(K)],
$$
which endows the coordinate ring $\CC[\R_{\agl}(K)]$ with a $\CC[t, t^{-1}]$-algebra structure. The discussion of Section \ref{sec:Alexander-module-top} and Corollary \ref{cor:Writinger-presentation-fox} implies that we can find a presentation of $\CC[\R_{\agl}(K)]$ as $\CC[t, t^{-1}]$-module that makes it isomorphic to the symmetric algebra of the Alexander module. However, we provide here a direct proof of this fact for completeness.

\begin{thm}\label{thm:coordinate-ring-agl1}
For any knot $K$, the coordinate ring $\CC[\R_{\agl}(K)]$ is isomorphic as $\CC[t, t^{-1}]$-algebra to the symmetric algebra $\Sym_{\CC[t,t^{-1}]} \left(H_1(K^c_\infty)\right)$ of the first Alexander module.
\end{thm}

\begin{proof}
Let $\Gamma_K$ be the knot group of $K$ and let $\Gamma_K' = \ker(\textup{Ab}_K)$ be the kernel of the abelianization map $\textup{Ab}_K: \Gamma_K \to H_1(K^c)$, which corresponds with the fundamental group of the cyclic cover $K^c_\infty$. In this manner, any representation $\rho: \Gamma_K \to \agl$ restricts to a representation $\rho|_{\Gamma_K'}: \Gamma_K'\to \agl$. But, furthermore, since the commutator subgroup of $\agl$ is abelian, such representation descends to the abelianization $\Gamma_K'/[\Gamma_K', \Gamma_K'] = H_1(K^c_\infty)$, leading to a representation
$$
    \textup{Ab}(\rho|_{\Gamma_K'}): H_1(K^c_\infty) \to \agl.
$$

In this way, if $\varpi: \agl \to \CC$ denotes the projection onto the off-diagonal entry, we construct a map
$$
\Phi_K: H_1(K^c_\infty) \longrightarrow \CC[\R_{\agl}(K)],
$$
that sends any $\gamma \in H_1(K^c_\infty)$ to the regular function $\textup{ev}_\gamma: \R_{\agl}(K) \to \CC$, defined as $\textup{ev}_\gamma(\rho) = \varpi \left(\textup{Ab}(\rho|_{\Gamma_K'})(\gamma)\right)$ for any $\rho \in \R_{\agl}(K)$. Notice that $\Phi_K$ is a monomorphism since $\R_{\agl}(K)$ is a separated algebraic variety.

It is straightforward to check that $\Phi_K$ is actually a morphism of $\CC[t, t^{-1}]$-modules and, therefore, it descends to a monomorphism of $\CC[t, t^{-1}]$-algebras
$$
\tilde{\Phi}_K: \Sym_{\CC[t,t^{-1}]} \left(H_1(K^c_\infty)\right) \longrightarrow  \CC[\R_{\agl}(K)].
$$
Furthermore, choose generators $\gamma_1, \ldots, \gamma_n$ of $\ker(\textup{Ab}_K)$ and complement them with a meridian $m$ to get a set of generators of $\Gamma_K$. Denote by
$$
    \rho(\gamma_i) = \begin{pmatrix}
        1 & 0 \\ a_i & t_i
    \end{pmatrix}, \qquad \rho(m) = \begin{pmatrix}
        1 & 0 \\ b & t
    \end{pmatrix},
$$
their images under a representation $\rho \in \R_{\agl}(K)$, for $1 \leq i \leq n$. 
We have that the functions $a_1, \ldots, a_n: \R_{\agl}(K) \to \CC$ generate the coordinate ring of $\R_{\agl}(K)$ as $\CC[t, t^{-1}]$-algebra, whereas the functions $t_1, \ldots, t_n$ and $b$ are algebraically dependent. 
This shows that $\tilde{\Phi}_K$ is surjective, as $a_i = \tilde{\Phi}_K(\gamma_i)$, for $1 \leq i \leq n$, by construction.
\end{proof}

\begin{ex}
Consider the trefoil knot $K$, whose Wirtinger presentation gives a description of the fundamental group of its complement
$$
    \Gamma_K = \langle x_1, x_2, x_3 \,|\, x_1 = x_2x_3x_2^{-1}, \, x_2 = x_3x_1x_3^{-1}, \, x_3 = x_1x_2x_1^{-1}\rangle.
$$
In this manner, the first Alexander module of $K$ is
$$
    H_1(K_\infty^c) = \frac{\CC[x_1, x_2, x_3]}{\left(tx_3+(1-t)x_2 - x_1, tx_1+(1-t)x_3 - x_2, tx_2+(1-t)x_1 - x_3\right)}
$$

On the other hand, for the representation variety $\R_{\agl}(K)$, given a representation $\rho: \Gamma_K \to \agl$, let us denote the images of the generators by
$$
    A_1 = \rho(x_1) = \begin{pmatrix}
        1 & 0 \\
        a_1 & t
    \end{pmatrix}, \quad A_2 = \rho(x_2) = \begin{pmatrix}
        1 & 0 \\
        a_2 & t
    \end{pmatrix}, \quad A_3 = \rho(x_3) = \begin{pmatrix}
        1 & 0 \\
        a_3 & t
    \end{pmatrix}.
$$
Then, a direct calculation shows that
$$
A_2A_3A_2^{-1} = \begin{pmatrix}
        1 & 0 \\
        ta_3+(1-t)a_2 & t
    \end{pmatrix}, \quad A_3A_1A_3^{-1} = \begin{pmatrix}
        1 & 0 \\
        ta_1+(1-t)a_3 & t
    \end{pmatrix},
$$
$$
    A_1A_2A_1^{-1} = \begin{pmatrix}
        1 & 0 \\
        ta_2+(1-t)a_1 & t
    \end{pmatrix}.
$$
Therefore, we have that 
$$
\R_{\agl}(K) = \left\{(a_1, a_2, a_3, t) \in \CC^3 \times \CC^*\,|\, a_1 = ta_3+(1-t)a_2, \, a_2 = ta_1+(1-t)a_3, \, a_3 = ta_1+(1-t)a_3\right\}.
$$
Thus, the coordinate ring $\CC[\R_{\agl}(K)]$ coincides with the symmetric algebra of the Alexander module by setting $a_i \mapsto x_i$ for $i = 1, 2, 3$.
\end{ex}

Furthermore, using the Alexander matrix $A_K$ to give a presentation of the Alexander module, we get the following result regarding the Alexander fibration, which in particular reproves Theorem \ref{thm:2.2}. 

\begin{cor} \label{cor:Alexanderfibration}
For any knot $K$ and any $t \in \CC^*$, the fibre $\mathcal{S}_t$ of the Alexander fibration  
$\pi_K: \R_{\agl}(K) \to \CC^*$ given in \eqref{eqn:representationfibration}, 
 is the set of solutions in $\CC^n$ of the linear system defined by the Alexander matrix of $K$, 
$$
\mathcal{S}_t=\pi^{-1}_K(t)=\lbrace a=(a_1,\ldots,a_n)\in \CC^n \,|\, A_K(t)\,a=0 \rbrace.
$$ 
In particular, the exceptional locus of $\pi_K$ is given by the roots of the Alexander polynomial.
\end{cor}


Explicitly, the chain of elementary ideals \eqref{eqn:elementaryideals} of the Alexander module yields a stratification of the exceptional locus (collections of points in $\CC^{\ast}$), since
\begin{equation}\label{eq:Alexander-stratification}
\emptyset=V(E_K^{n}) \subset V(E_K^{n-1}) \subset \ldots \subset V(E_K^2)\subset V(E_K^1)=V(\Delta_K) \subset V(E_K^0)=\CC^{\ast},
\end{equation}
where $V(E_K^k)$ is the zero locus of the $k$-th Alexander polynomial $\Delta_K^k$.

From the definition of the elementary ideals we obtain
\begin{cor} \label{cor:excfiber}
Let $t\in \CC^{\ast}$ such that $t\in V(E_K^k) - V(E_K^{k+1})$. Then $\pi_K^{-1}(t) \cong \CC^{k+1}$ for all $k=0,1,\ldots, n-1$.
\end{cor}

\begin{rem}
From Corollary \ref{cor:excfiber}, we can explicitly compute the virtual class of the $\agl$-representa\-tion variety of a knot in the Grothendieck ring of algebraic varieties. 
We proceed by stratifying $\R_{\agl}(K)$ in terms of the type of representation. Notice that the abelian representations give a global contribution of $\LL(\LL-1)$, where $\LL = [\mathbb{A}^1_\CC]$ is the Lefschetz motive. Over $\CC^{\ast}-V(\Delta_K)$, the only 
$\agl$-representations are abelian, thus they contribute $\LL\left([V(E_K^0)]-[V(E_K^1]\right)= \LL(\LL-1-\sum_k \left| S_K^k \right|)$, where $S_K^k=V(E_K^k)-V(E_K^{k+1})$. Besides, Corollary \ref{cor:excfiber} provides a natural stratification of the singular locus. If $t\in V(E_K^k) -  V(E_K^{k+1})$, we obtain a contribution of $\left| S_K^k \right| \LL^{k+1}$, already accounting for abelian representations. Thus, the virtual class is
$$
[\R_{\agl}(K)]= \LL\left(\LL-1-\sum_{k=1}^{n-1} \left| S_K^k \right| \right)+ \sum_{k=1}^{n-1} \left| S_K^k\right| \LL^{k+1}=\LL(\LL-1)+ \sum_{k=1}^{n-1} \left| S_K^k\right| (\LL^{k+1}-\LL) .
$$
\end{rem}





\subsection{The Alexander fibration as a singular vector bundle} \label{sec:alexsingularvectorbundle}

Let $K \subset S^3$ be a knot. In this section, we are going to show that the Alexander fibration
$$
    \pi_K: \R_{\agl}(K) \to \R_{\CC^*}(K) = \CC^*
$$
can be seen as a singular vector bundle on $\CC^*$, and we recover the Alexander module from it.

To do so, recall that a coherent sheaf $\mathcal{M}$ over an algebraic variety $X$ is a finitely presented sheaf of $\mathcal{O}_X$-modules, that is
it can be written as the cokernel of a map $\mathcal{O}_X^m \stackrel{g}{\longrightarrow} \mathcal{O}_X^n$.
We can construct an associated geometric space as
$$
    E = \mathbf{Spec}(\Sym_{\mathcal{O}_X}(\mathcal{M})),
$$
where $\Sym_{\mathcal{O}_X}(\mathcal{M})$ is the symmetric algebra of $\mathcal{M}$ as $\mathcal{O}_X$-module and $\mathbf{Spec}$ is the relative spectrum. Observe that $E$ comes equipped with a regular morphism $\pi: E \to X$.
In the case that $\mathcal{M}$ is a locally free sheaf of rank $n$, $E$ is a vector bundle of rank $n$ over $X$. 
However, if $\mathcal{M}$ is not locally free, then its stalks have varying (finite) rank.
As $\mathcal{M}$ is a coherent sheaf, dualizing the exact sequence $\mathcal{O}_X^m \to \mathcal{O}_X^n \to \mathcal{M} \to 0$,
we see that $E$ is the kernel of a vector bundle homomorphism $g:E_1\to E_2$ between two vector bundles $E_1,E_2$ over $X$. This means that the fibre
of $E$ over $p\in X$, $E_p=\pi^{-1}(p)=\ker(g_p:E_{1,p}\to E_{2,p})$ is a complex vector space, but the dimensions of $E_p$ depend on $p\in X$.
We baptise these objects as singular vector bundles, according to the following definition.

\begin{defn}\label{defn:singular-vector-bundle}
    A \emph{singular vector bundle} over an algebraic variety $X$ is a variety $E$ with a regular morphism $\pi: E \to X$ such that there exists an open covering $\{U_\alpha\}$ of $X$ and injective morphisms $\psi_\alpha: \pi^{-1}(U_\alpha) \to U_\alpha \times \mathbb{C}^n$ 
    making the following diagram commutative
    $$
\xymatrix{
E_\alpha:=\pi^{-1}(U_\alpha)  \ar[rr]^{\psi_\alpha} \ar[rd] & & U_\alpha \times \mathbb{C}^n\ar[ld]^{\pi} \\
& U_\alpha &
}
    $$    
    and such that the changes of charts are fibrewise linear. 
    A morphism between singular vector bundles $\pi: E \to X$ and $\pi': E' \to X'$ is a pair $(f, \tilde{f})$ of regular morphisms $f: X \to X'$ and $\tilde{f}: E \to E'$ such that the following diagram commutes
        $$
\xymatrix{
E \ar[r]^{\tilde{f}}\ar[d]_{\pi} & E'\ar[d]^{\pi'} \\
X \ar[r]_{f} & X'
}
    $$    
    and satisfying that $\tilde{f}$ is fibrewise linear.
\end{defn}


\begin{rem}
    The term `singular' vector bundles is not standard in the literature, and sometimes they are just called \emph{geometric coherent sheaves}. They are completely analogous to regular vector bundles and, in particular, they admit the typical constructions such as fibrewise direct sums and tensor products. But they also admit kernels, actually forming an abelian category. 
\end{rem}

It is worth mentioning that the functor $\mathcal{M} \mapsto E = \mathbf{Spec}(\Sym_{\mathcal{O}_X}(\mathcal{M}))$ from coherent sheaves to singular vector bundles is inverse to the functor that takes a singular vector bundle $E$ and considers its sheaf of sections, which is a coherent sheaf.

Returning to our setting, let us consider the sheaf of sections of the Alexander fibration $\pi_K: \R_{\agl}(K) \to \CC^*$, denoted by $\mathcal{M}_K$ and given by
$$
	\mathcal{M}_K(U) = \left\{s: U \to \R_{\agl}(K) \textrm{ regular}\,|\, \pi_K \circ s = \id_U \right\},
$$
for any (Zariski) open set $U \subset \CC^*$. It is straightforward to check that $\mathcal{M}_K$ is a sheaf of complex vector spaces. Furthermore, since the fibre $\pi_K^{-1}(t)$ is a complex vector space for each $t \in \CC^*$, this sheaf has a natural $\mathcal{O}_{\CC^*}$-module structure, where $\mathcal{O}_{\CC^*}$ is the structure sheaf of regular functions on $\CC^*$. 

\begin{prop}\label{prop:coherent}
For any knot $K$, $\mathcal{M}_K$ is a coherent sheaf on $\CC^*$.
\end{prop}

\begin{proof}
On an affine open set $U = \Spec A \subset \CC^*$ we have
$$
    \mathcal{M}_K(\Spec A) = \left\{s^*: \CC[\R_{\agl}(K)] \to A \,|\, s^* \circ \pi_K^* = \id_A \right\} \subset \Hom_{\CC\textrm{-\textbf{Alg}}}(\CC[\R_{\agl}(K)], A),
$$
where $\CC[\R_{\agl}(K)]$ is the coordinate ring of $\R_{\agl}(K)$. Since the morphism $\pi_K^*: A \to \CC[\R_{\agl}(K)]$ is determined by the inclusion $\CC[t,t^{-1}] \hookrightarrow \CC[\R_{\agl}(K)]$, this is a finitely generated $A$-module, implying that $\mathcal{M}_K$ is coherent.
\end{proof}

\begin{cor}  
    For any knot $K$, the $\agl$-representation variety coincides with the singular vector bundle $\R_\agl(K) = \mathbf{Spec}(\Sym_{\mathcal{O}_{\CC^*}}(\mathcal{M}_K))$ associated to the coherent sheaf $\mathcal{M}_K$.
\end{cor}


Furthermore, as shown in \cite[Proposition II.5.1]{Hartshorne}, every coherent sheaf is uniquely determined by the finitely generated module $M_K = \mathcal{M}_K(\CC^*)$ of global sections over the ring $\mathcal{O}_{\CC^*}(\CC^*) = \CC[t, t^{-1}]$. In fact, since $M_K$ is isomorphic to the ring of regular functions on $\R_{\agl}(K)$, by Theorem \ref{thm:coordinate-ring-agl1} we get the following result.


\begin{thm}\label{thm:Alexander-coherent}
$M_K$ is isomorphic to the symmetric algebra of the Alexander module of the knot $K$, as $\CC[t, t^{-1}]$-algebras.
\end{thm}


\begin{rem}
    The coherent sheaf $\mathcal{M}_K$ has been previously studied by Hironaka in \cite{hironaka1997alexander} in a different context, with an ad hoc construction by means of $1$-cocycles. In this direction, the Alexander fibration (\ref{eq:Alexander-stratification}) coincides with the one studied in \cite{hironaka1997alexander}.  
\end{rem}


\subsection{Alexander data in a broader context}\label{sec:alexander-broader}

Consider any finitely generated group $\Gamma$. In this context, we can also consider the Alexander fibration
$$
    \pi_\Gamma: \R_{\agl}(\Gamma) \to \R_{\CC^*}(\G)
$$
induced by the projection homomorphism $\mu: \agl = \CC \rtimes\CC^* \to \CC^*$. Indeed, fixed $t \in \CC^*$, we have that $\mu^{-1}(t) \subset \agl$ is naturally isomorphic to $\CC$, 
and in particular it has a natural vector space structure. This implies that, for any $\tau \in \R_{\CC^*}(\Gamma)$, the fibre $\pi_\Gamma^{-1}(\tau) \subset \R_{\agl}(\Gamma)$ is naturally a vector space. Indeed, given $\rho_1, \rho_2 \in \pi_\Gamma^{-1}(\tau)$ and $\lambda \in \CC$, the representation $\rho_1 + \lambda \rho_2 \in \pi_\Gamma^{-1}(\tau)$ is given, for $\gamma \in \Gamma$, by
$$
    (\rho_1 + \lambda \rho_2)(\gamma) = (a_1 + \lambda a_2, t),
$$
where $\rho_1(\gamma) = (a_1, t) \in \agl$ and $\rho_2(\gamma) = (a_2, t) \in \agl$. Notice that the last components of $\rho_1(\gamma)$ and $\rho_2(\gamma)$ agree since $t = \tau(\gamma)$.

This fact anticipates that the Alexander fibration $\pi_\Gamma$ exhibits $\R_{\agl}(\Gamma)$ as a singular vector bundle over $\R_{\CC^*}(\Gamma)$. To be precise, let us consider the sheaf $\mathcal{M}_\Gamma$ over $\R_{\CC^*}(\Gamma)$ of sections of $\pi_\Gamma$, that is
$$
	\mathcal{M}_\Gamma(U) = \left\{s: U \to \R_{\agl}(\Gamma) \textrm{ regular}\,|\, \pi_\Gamma \circ s = \id_U \right\},
$$
for any open set $U \subset \R_{\CC^*}(\Gamma)$. Since each fibre of $\pi_\Gamma$ is a vector space, $\mathcal{M}_\Gamma$ is a $\mathcal{O}_{\R_{\CC^*}(\Gamma)}$-module. Furthermore, arguing completely analogously to the proof of Proposition \ref{prop:coherent}, we get the following result.

\begin{prop}\label{prop:coherent-general}
    For any finitely generated group $\Gamma$, the sheaf $\mathcal{M}_\Gamma$ is coherent.
\end{prop}

Even more, we have that $\R_{\agl}(\Gamma) = \textbf{Spec}(\Sym_{\mathcal{O}_{\R_{\CC^*}(\Gamma)}}(\mathcal{M}_\Gamma))$, showing that the Alexander fibration $\pi_K: \R_{\agl}(\Gamma) \to \R_{\CC^*}(\Gamma)$ is a singular vector bundle.

Now, let us focus on the case where $\Gamma$ is a link group. To be precise, we fix a link $L \subset S^3$ and consider the corresponding link group $\Gamma_L = \pi_1(L^c)$, where $L^c = S^3 - L$ is the link complement. With this, for any complex affine algebraic group $G$, we can form the associated representation variety
$$
    \R_G(L) := \R_{G}(\G_L) = \Hom(\Gamma_L, G).
$$

For the Alexander fibration, observe that in this case we have that $H_1(L^c) = \ZZ^s$, where $s$ is the number of components of $L$. Therefore, $\R_{\CC^*}(L) = \Hom(H_1(L^c), \CC^*) = (\CC^*)^s$ and, in this situation, the Alexander fibration leads to a singular vector bundle
\begin{equation} \label{eqn:representationfibration-link}
    \pi_L: \R_{\agl}(L) \longrightarrow (\CC^{\ast})^s.
\end{equation}

Using this information, we can stratify $(\CC^*)^s$ according to the rank of the fibres, giving a filtration into closed subvarieties
\begin{equation}\label{eq:Alexander-stratification-link}
\emptyset=Z_L^n \subset Z_L^{n-1} \subset \ldots \subset Z_L^2 \subset Z_L^1 \subset 
Z_L^0 = (\CC^*)^s\, .
\end{equation}
In this manner, we can also define the multivariable Alexander polynomial $\Delta_L(t_1, \ldots, t_s) 
\in \CC[t_1, \ldots, t_s]$ to be the polynomial defining the hypersurface $Z_L^1$. Unlike the knot case, there is no analogue of $\Delta_L^k$ for $k>1$ because the subvarieties $Z_L^k$ are not defined by a single polynomial.

\section{Quandles} \label{sec:quandles}

In this section, we recall the theory of quandles, with a view towards the knot quandle as defined by Joyce in \cite{Joycequandle}. Roughly speaking, a quandle captures the operation of conjugation in a group, forgetting about the rest of the group structure. As we will show, the quandle theory is very effective to describe the Alexander fibration.

\begin{defn}
A \emph{quandle} is a set $Q$ equipped with a binary operation $\triangleright: Q\times Q \longrightarrow Q$ satisfying the following properties:
\begin{itemize}
\item[(P1)] $x\triangleright x= x$, for all $x\in Q$.
\item[(P2)] For all $x,z\in Q$, there exists a unique $y\in Q$ such that $x\triangleright y=z$. This defines another operation $\triangleleft$ given by $z \triangleleft x=y$.
\item[(P3)] $ x\triangleright (y\triangleright z) = (x\triangleright y)\triangleright (x\triangleright z)$, for all $x,y,z\in Q$.
\end{itemize}
\end{defn}
A quandle is said to be complex linear algebraic if it is a quandle object in the category of complex affine varieties, i.e.\ if $Q$ is a complex affine variety and the operations $\triangleright$ and $\triangleleft$ are regular morphisms.

A quandle homomorphism between quandles $(Q,\triangleright_Q),(Q',\triangleright_{Q'})$ is a map $f:Q\rightarrow Q$ preserving the quandle operation, i.e.\ $f(x\triangleright_Q y)=f(x)\triangleright_{Q'}f(y)$ for all $x,y \in Q$. This defines a category of quandles, denoted by $\Qdl$.

\begin{rem}
Given any set $S$, the \emph{trivial quandle} is $S$ with the binary operation $x\triangleright y=y$ for all $x,y\in S$.
\end{rem}

A way of constructing a quandle is to start with a group $G$. Then, on the set of elements of $G$, we define a quandle operation using conjugation by
$$
	a\triangleright b= aba^{-1}, \qquad \textup{for }a, b \in G.
$$
In this case, we have that $b \triangleleft a = a^{-1}ba$. This quandle structure on $G$ is usually referred to as the \emph{group quandle} or the \emph{conjugation quandle}, and denoted by $\Conj(G)$. Notice that, when $G$ is abelian, the quandle structure is the trivial one. Reciprocally, given a quandle $Q$, we can define the group associated to $Q$ as
$$
	\As(Q) = \langle a \in Q \,|\, a\triangleright b= aba^{-1}, \textrm{ for $a, b \in Q$}\rangle.
$$
In other words, $\As(Q)$ is the group generated by the elements of $Q$ subject to the relations $a\triangleright b= aba^{-1}$.
These operations actually define an adjunction between the category $\textbf{Grp}$ of groups and $\Qdl$ (c.f.\ \cite[Lemma 1.6]{racks})
\begin{equation}\label{eq:adj}
\xymatrix{
\textbf{Grp} \ar@/^1.0pc/[rr]^{\Conj} && \Qdl, \ar@/^1.0pc/[ll]^{\textrm{As}}
}
\end{equation}
Explicitly, $\Conj$ is the left adjoint and $\As$ is the right adjoint, that is
$$
    \Hom_{\textbf{Grp}}(\As(Q), G) = \Hom_{\Qdl}(Q, \Conj(G)),  
$$
for any quandle $Q$ and group $G$. In particular, the unit of the adjunction induces a map $G \to \As(\Conj(G))$ for any group $G$, but in general this is not an isomorphism.

\begin{rem}\label{rmk:free-quandle}
Given any set $S$, the \emph{free quandle} on $S$ as the set $\FQ_S \subset \Free{S}$, where $\Free{S}$ denotes the free group generated by $S$, 
consisting of all conjugates in $\Free{S}$ of the elements of $S$, that is $\FQ_S=\{ w s w^{-1} | w\in \Free{S}, s\in S\}$. The 
quandle operation is defined by $a \triangleright b= aba^{-1}$. This is a subquandle of $\Conj(\Free{S})$, and it is straightforward to see
that $\As(\FQ_S)=\Free{S}$. 

The free quandle $\FQ_S$ satisfies the universal property that $\Hom_{\Qdl}(\FQ_S, Q) = 
\Hom_{\textbf{Set}}(S, Q)$ for any quandle $Q$. The free quandle on a set of $n$ elements will be denoted by $\FQ_n$.
\end{rem}

\begin{rem}\label{rmk:quandle-alexander}
Given any $t \in \CC^* = \CC - \{0\}$, the \emph{$t$-Alexander quandle} is $\AQ{t} = (\CC, \triangleright_t)$ with the operations
$$
	a\triangleright_t b=tb+(1-t)a, \qquad b \triangleleft_t a = t^{-1}b + (1-t^{-1})a,
$$
for $a, b \in \CC$.

 Consider the conjugation quandle $\Conj(\agl)$ associated to $\agl$. For any $t \in \CC^*$, it is straightforward to check that the map from the $t$-Alexander quandle
    $$
        \AQ{t} \to \Conj(\agl), \qquad a \mapsto \begin{pmatrix}
            1 & 0 \\ a & t
        \end{pmatrix},
    $$
    defines an injective quandle morphism. It thus identifies $\AQ{t}$ with the subquandle of $\Conj(\agl)$ of matrices with diagonal entry equal to $t$.
\end{rem}


\subsection{The knot quandle}

Let us recall the definition of a knot quandle and its multiple incarnations, as introduced by Joyce \cite{Joycequandle} and Matveev \cite{Matveev}. 

Let $M$ be a connected oriented manifold and let $L \subset M$ be an oriented submanifold of codimension two. Typically, $L$ will be a link in $S^3$ or a collection of points on a surface. The orientations on $L$ and $M$ induce an orientation on the normal bundle of $L$ and thus on a tubular neighbourhood $N_L \subset M$ of $L$. Denote by $E_L=M-N_L$ the exterior of this neighbourhood. Fixing a base point $x_0\in E_L$, we define $\Gamma_L$ as the set of homotopy classes of paths in $E_L$ from $x_0$ to $\partial N_L$, where homotopies are required to keep the initial point $x_0$ fixed while allowing the endpoint to move freely in $\partial N_L$. Furthermore, note that any point of $\partial N_L$ belongs to a unique meridian circle of the normal circle bundle, so any $\alpha \in \Gamma_L$ defines a unique loop $m_\alpha$ in $\partial N_L$ based at the endpoint of $\alpha$ and positively oriented with respect to the orientation of $N_L$. This allows us to define an operation on $\Gamma_L$ that, given paths $\alpha, \beta \in \Gamma_L$, assigns
$$
\alpha \triangleright \beta = \alpha \star m_{\alpha} \star \bar{ \alpha}\star \beta,
$$
where $\star$ denotes the standard concatenation of paths, $\bar{\alpha}$ denotes the reverse path to $\alpha$, and we choose the same path representative
for $\alpha$ and $\bar{\alpha}$, and the path $m_{\alpha}$ at the endpoint of $\alpha$. 
This operation endows $Q_L = (\Gamma_L,\triangleright)$ with a quandle structure, called the \emph{quandle of $L$}. Notice that $Q_L$ only depends on the chosen basepoint up to isomorphism.

There is a natural map 
\begin{equation}\label{eqn:as-QL}
 \begin{array}{ccc}
  Q_L & \longrightarrow & \pi_1(L^c)=\pi_1(M-N_L) \\
  \alpha & \longmapsto & \alpha\star m_\alpha\star \bar{\alpha},
  \end{array}
  \end{equation}
that shows that $\As(Q_L) =\pi_1(L^c)$.
\begin{rem}
    The construction of the quandle $L$ is functorial in the following sense. Let $L \subset M$ as above, and suppose that $N \subset M$ is an oriented submanifold that intersects $L$ transversely. Then $L \cap N$ is a submanifold of $N$ of codimension two with a naturally induced orientation. Then the quandle $Q_{L \cap N}$ is well defined and the inclusion induces a quandle morphism $Q_{L \cap N} \to Q_L$.
\end{rem}

In the case of links $L \subset S^3$, it is possible to give a Wirtinger-type presentation of the quandle of $L$. Given a regular projection of $L$, it defines a link diagram $D$ with set of arcs $A_D = \{x_1, \ldots, x_n\}$. For each positive crossing $L_+$ in Figure 
\ref{fig:wirtinger}, define the relation $x_j\triangleright_D x_i=x_{i+1}$, and for each negative crossing $L_-$ define the relation
$x_i\triangleleft_D x_j=x_{i+1}$. 
In \cite{fenn1992racks}, it is proven that $\langle A_D \mid \triangleright_D \rangle$ is a presentation of $Q_L$. In particular, it proves that $Q_L$ is finitely generated. This presentation is sometimes known as the \emph{algebraic quandle of $L$}.

Indeed, relating the algebraic quandle description with the Wirtinger presentation of Section \ref{sec:2}, we get that the associated group of the link quandle is the fundamental group of its complement
$$
     \As(Q_L) = \pi_1(S^3-L)=\pi_1(L^c),
$$
agreeing with (\ref{eqn:as-QL}).

One of the main reasons for the importance of quandles in knot theory is the following result by Joyce.

\begin{thm}[\cite{Joycequandle}]
The knot quandle $Q_K$ is a complete knot invariant up to orientation.
\end{thm}

\begin{ex} \label{ex:quandle-free-product}
 Suppose that $L$ is a link formed by two disjoint sublinks $L=L_1\sqcup L_2$, that is, in the planar diagram there are no
crossings between strands of $L_1$ and $L_2$. Then the description of the link quandle implies that $Q_L=Q_{L_1}\star Q_{L_2}$,
where the free product of quandles is defined as the set of all conjugates of elements of $Q_{L_1}$ or $Q_{L_2}$ inside the
free product of groups $\As(Q_{L_1})\star \As(Q_{L_2})$. 
Note that $\Hom_{\Qdl}(Q_{L_1}\star Q_{L_2}, X) \cong \Hom_{\Qdl}(Q_{L_1}, X) \x \Hom_{\Qdl}(Q_{L_2}, X)$ for any quandle $X$.
\end{ex}

There is an analogous representation theory of quandles parallel to the one of representation varieties of groups. Given a finitely generated quandle $Q$, i.e.\ one that admits a surjective homomorphism $\FQ_n \to Q$ for some $n \geq 1$, and a complex linear algebraic quandle $X$, we can consider the quandle representation variety
$$
	\R_X(Q) = \Hom_{\Qdl}(Q, X).
$$
In the same vein as representation varieties of groups, the quandle representation variety $\R_X(Q)$ inherits a natural structure of an affine algebraic variety.

\begin{rem}
    In the case that $X = \Conj(G)$ is the conjugation quandle of an algebraic group $G$, we have a natural equivalence $\R_{\Conj(G)}(Q) = \Hom_{\Qdl}(Q, \Conj(G)) \cong \Hom_{\textbf{Grp}}(\As(Q), G)$ so the quandle representation variety is isomorphic to the usual representation variety $\R_{G}(\As(Q))$.
\end{rem}

These quandle representation varieties are particularly relevant in the context of the Alexander fibration \eqref{eqn:representationfibration}, where we can identify the fibre using quandle representations. Given a knot $K \subset S^3$, let us denote by 
 $$
\R_{X}(K) := \Hom_{\Qdl}(Q_K, X)
$$ 
the associated quandle representation variety to the knot quandle $Q_K$. 

\begin{prop}
For any $t\in \CC^{\ast}$, the fibre $\mathcal{S}_t$ over $t$ of the Alexander fibration is the $\AQ{t}$-quandle representation variety $\R_{\AQ{t}}(K) = \Hom_\Qdl(Q_K, \AQ{t})$.
\end{prop}

\begin{proof}
By Remark \ref{rmk:quandle-alexander}, $\AQ{t}$ is the subquandle of $\Conj(\agl)$ of matrices with diagonal entry equal to $t$. This defines an inclusion
$$
    \R_{\AQ{t}}(K) = \Hom_{\Qdl}(Q_K, \AQ{t}) \hookrightarrow \Hom_{\Qdl}(Q_K, \Conj(\agl)) = \R_\agl(K),
$$
where we have used that $\As(Q_K) = \G_K$. Indeed, the image of this inclusion is precisely the set of representations $\rho \in \R_\agl(K) = \Hom_{\Qdl}(Q_K, \Conj(\agl))$ such that $\pi_K(\rho) = t$, showing that it is exactly the $t$-fibre of the Alexander fibration.
\end{proof}



\begin{cor} \label{cor:stalks}
For $t \in \CC^*$, the fibre over $t$ of the Alexander fibration $\pi_K: \R_{\agl}(K) \to \R_{\CC^*}(K) = \CC^*$ is $\R_{\AQ{t}}(K)$.
\end{cor}

\begin{rem}
    In the language of the work of Hironaka \cite{hironaka1997alexander} for the coherent sheaf $\mathcal{M}_K$ of regular sections of the Alexander fibration, the fibre of (the associated singular vector bundle of) this sheaf can be understood in terms of group cohomology. To be precise, \cite{hironaka1997alexander} 
    studies a coherent sheaf over $\R_{\CC^*}(\Gamma) = \Hom(\Gamma, \CC^*)$ whose fibre over $\rho \in \R_{\CC^*}(\Gamma)$ is given by the $1$-cocycles
    $$
        Z^1(\Gamma, \rho) = \left\{f: \Gamma \to \CC \,|\, f(\gamma_1\gamma_2) = f(\gamma_1) + \rho(\gamma_1)f(\gamma_2)\right\}.
    $$
    Indeed, if $\Gamma = \Gamma_K$ for a knot $K$, then $\rho \in \R_{\CC^*}(\Gamma)$ is determined by the image $t \in \CC^*$ of the standard generators $x_1, \ldots, x_n$ of $\Gamma_K$, and thus we get that $Z^1(\Gamma_K, \rho)$ is the collection of functions $f: \Gamma_K \to \CC$ such that
    $$
        f(x_ix_j) = f(x_i) + tf(x_j),
    $$
    for all $1 \leq i,j\leq n$ or, in terms of conjugation
    $$
        f(x_ix_jx_i^{-1}) = (1-t)f(x_i) + tf(x_j),
    $$
    which coincides with the quandle representation variety $\R_{\AQ{t}}(K)$.
\end{rem}

\begin{rem}
    For a general algebraic group $G$ with Lie algebra $\mathfrak{g}$ and finitely generated group $\Gamma$, the tangent space $T_{\rho} \R_G(\Gamma)$ to $\R_G(\Gamma)$ at a point $\rho \in \R_G(\Gamma)$ is given by the group cohomology $1$-cocycles for the adjoint representation $\Ad_\rho = \Ad \circ \rho: \Gamma \to G \to \textup{GL}(\mathfrak{g})$ (see \cite{weil})
    $$
        T_{\rho} \R_G(\Gamma) = Z^1(\Gamma, \Ad_\rho) = \left\{f: \Gamma \to \mathfrak{g}\,|\, f(\gamma_1\gamma_2) = f(\gamma_1) + \Ad_{\rho}(\gamma_1)f(\gamma_2)\right\}.
    $$
    
    In this manner, the coherent sheaf $\mathcal{M}_K$ can be seen as a sort of ``twisted'' tangent sheaf for $\R_{\CC^*}(K)$ where, over $\rho \in \R_{\CC^*}(K)$, we use the representation $\rho$ itself instead of the adjoint $\Ad_\rho$.
\end{rem}

\begin{rem}\label{rmk:fibrevectorspace-link}
 Most of the discussion of this section extends verbatim to links $L \subset S^3$, using the link quandle $Q_L$ and its link group $\G_L = \pi_1(L^c)$, which coincides with $\As(Q_L)$. However, despite the parallelism with a knot, in this case the fibre $\pi_L^{-1}(t_1, \ldots, t_s) \subset \R_{\agl}(L) = (\CC^*)^s$ of the Alexander fibration does not have in general a simple interpretation in terms of representations onto Alexander quandles, since different diagonal entries are assigned to different meridians. However, the `diagonal' fibre $\pi_L^{-1}(t, t, \ldots, t)$ does recover the quandle interpretation being isomorphic to $\R_{\AQ{t}}(Q_L)$.
     
\end{rem}

\section{Tangles and Alexander quandles} \label{sec:tanglesandalexquandles}

\subsection{The tangle category} \label{sec:tangle-cat}

We outline the main definitions and properties of the tangle category; for a detailed exposition we refer to \cite{Kasselquantumgroups}. For any integer $n \geq 1$, let $[n]$ denote the set $\{1,\ldots,n \}$ and set $[0] = \emptyset$. A tangle $T$ is the union of a finite number of disjoint oriented polygonal arcs in $\mathbb{R}^2 \times [0,1]$ such that its boundary satisfies the condition
$$
\partial T=T \cap\left(\mathbb{R}^2 \times\{0,1\}\right)=([k] \times\{0\} \times\{0\}) \cup([\ell] \times\{0\} \times\{1\}),
$$
for some $k, \ell \geq 0$. Notice that in this definition we also allow closed polygonal loops with no endpoints. 


Tangles naturally form a category $\cT$ as follows:
\begin{itemize}
\item \emph{Objects}: finite sequences of $\pm$ signs, possibly empty. Given any sequence $\epsilon = (\epsilon_1, \ldots, \epsilon_n)$ with $\epsilon_i = \pm$, we shall write $-\epsilon = (-\epsilon_1, \ldots, -\epsilon_n)$ for the sequence where all signs have been flipped, and $|\epsilon| = n$ for its length. Equivalently, an object $\epsilon$ can be seen as the complementary surface $\epsilon^c = \mathbb{R}^2 - \left([|\epsilon|] \times \{0\}\right)$ with a choice of orientations for the normal bundle around the punctures, according to the signs $\epsilon_i$.
\item \emph{Morphisms}: isotopy classes of oriented tangles. Any oriented tangle $T$ has two associated objects: its source $s(T)$, and target $b(T)$, which are the collections of signed points that belong to $\partial T\cap (\mathbb{R}^2\times \{0\})$ and $\partial T\cap (\mathbb{R}^2\times \{1\})$ respectively, so that the tangle defines a morphism between them. We will consider source and target as sequences of signed points immersed in $\mathbb{R}^2$ or $T$ as above whenever it is needed. Given a tangle $T$, the sign of each point of the sequence $s(L)$ is positive (resp. negative) if the point is an origin (resp. endpoint) of $T$, and conversely for $b(T)$.

Given any object $\epsilon$ of $\cT$, the identity morphism $\id_{\epsilon}$ is the tangle formed by the union of intervals $[|\epsilon|]\times \{0\}\times [0,1]$, where the orientation of the intervals is determined so that $s(\id_{\epsilon})=b(\id_{\epsilon})=\epsilon$. In particular, we will write $\id_+=\,\uparrow$ and $\id_-=\,\downarrow\,$.
\item \emph{Composition}: The composition of two tangles, $T$ and $T'$ is the isotopy class of the tangle that is obtained by placing $T'$ on top of $T$, which is well defined whenever $b(T)=s(T')$.
\end{itemize}

Furthermore, $\cT$ is a tensor category, where the tensor product of two objects is obtained concatenating sequences: if $\epsilon=(\epsilon_1,\ldots,\epsilon_k)$ and $\varepsilon=(\varepsilon_1,\ldots,\varepsilon_l)$, then
$$
\epsilon\otimes \varepsilon=(\epsilon_1,\ldots,\epsilon_k,\varepsilon_1,\ldots,\varepsilon_l),
$$
and thus the empty sequence $\emptyset$ is the unit. The tensor product of morphisms $T$ and $T'$ is the isotopy class of the tangle that is obtained placing $T'$ to the right of $T$. This monoidal structure is braided where the braiding is given by the tangle $B_{\epsilon, \varepsilon}$ depicted in Figure \ref{fig:braiding}.
Observe that $B_{\varepsilon, \epsilon} \circ B_{\epsilon, \varepsilon} \neq \id_{\epsilon \otimes \varepsilon}$, so the monoidal structure is not symmetric.

\begin{figure}[h]
    \centering
    \includegraphics[width=0.25\linewidth]{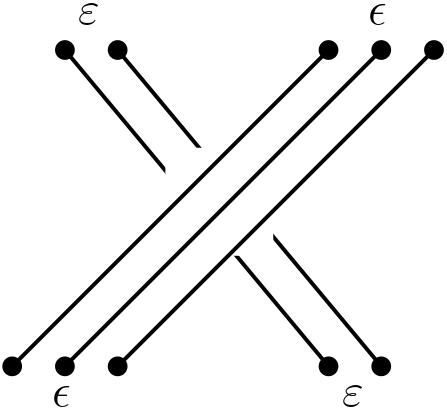}
    \caption{Braiding $B_{\epsilon, \varepsilon}: \epsilon \otimes \varepsilon \to \varepsilon \otimes \epsilon$.}
    \label{fig:braiding}
\end{figure}

Tangles can be represented using tangle diagrams, which are tangle projections to $\mathbb{R}\times [0,1]$ so that every crossing is simple. As it occurs for links and diagrams, two tangle diagrams represent isotopic tangles if one is obtained from the other by a sequence of Reidemeister transformations and 
isotopies of diagrams.

\begin{rem}\label{rem:relations-tangles}
There exists a very explicit description of the tensor category $\cT$ in terms of generators and relations, as presented in \cite{Kasselquantumgroups}. The tensor category $\cT$ admits a presentation as a monoidal category generated by the six elementary morphisms depicted in Figure \ref{fig:tangle-gens}, 
\begin{equation}\label{eqn:generatorstanglecat}
\upcurvearrowright, \upcurvearrowleft, \curvearrowright, \curvearrowleft, X_{+}, X_{-},
\end{equation}
subject to a collection of relations, detailed in \cite[Theorem XII.2.2]{Kasselquantumgroups}, which arise from isotopies of diagrams and Reidemeister transformations.
\begin{figure}[h]
    \centering
    \includegraphics[width=0.42\linewidth]{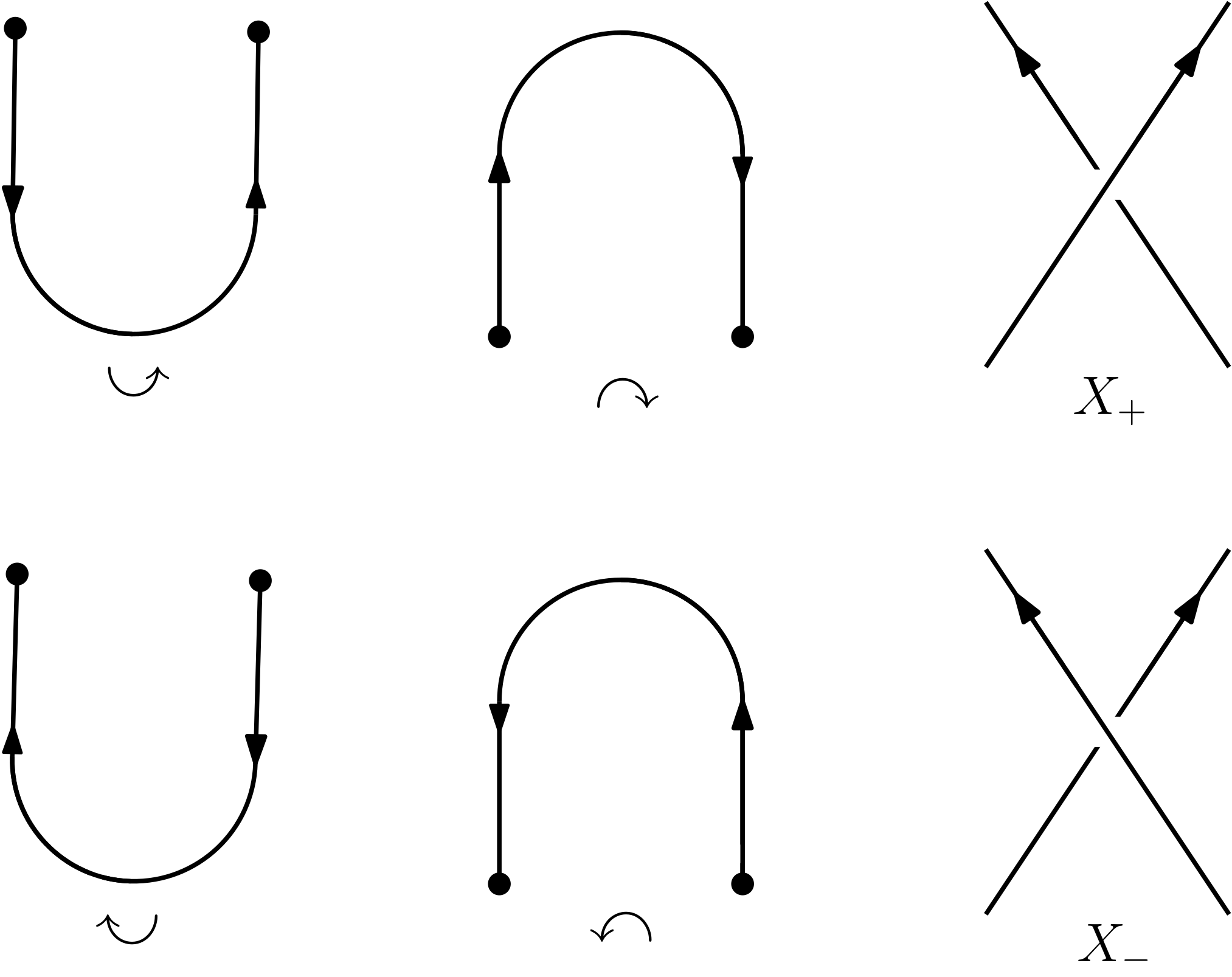}
    \caption{Generators of the tangle category $\cT$.}
    \label{fig:tangle-gens}
\end{figure}
\end{rem}


\subsection{The category $\Span(\Vect_{\CC})$}\label{sec:span-cat}
We outline now the main features of the category $\Span(\Vect_{\CC})$, where $\Vect_{\CC}$ denotes the category of finite dimensional complex vector spaces. We have
\begin{itemize}
\item \emph{Objects}: they are the same as in $\Vect_{\CC}$, finite dimensional complex vector spaces.
\item \emph{Morphisms}: given $X,Y\in \Vect_{\CC}$, a morphism $X\longrightarrow Y$ is given by an equivalence class of tuples $(Z, f, g)$ of linear maps
\begin{equation} \label{eqn:linearspan}
\xymatrix{
& Z \ar[dl]_{f} \ar[dr]^{g} & \\
X & & Y
}
\end{equation}
where $Z$ is another complex vector space. Such a tuple is usually referred to as a \emph{span} of vector spaces. Spans break asymmetries between source and target: a span with $f=\id$ is just a linear map from $X$ to $Y$ and a span with $g=\id$ is a linear map from $Y$ to $X$. Two spans $X \stackrel{f}{\longleftarrow} Z \stackrel{g}{\longrightarrow} Y$ and $X \stackrel{f'}{\longleftarrow} Z' \stackrel{g'}{\longrightarrow} Y$ are declared to be equivalent if there exists an isomorphism $\varphi: Z \to Z'$ such that the following diagram commutes

$$
\xymatrix{
	& Z \ar[ld]_{f} \ar[rd]^{g} \ar@{--{>}}[dd]^{\varphi} & \\
	X & & Y \\
	& Z' \ar[lu]^{f'} \ar[ru]_{g'} & 
}
$$

\item \emph{Composition}: The composition of two morphisms $(f_1,g_1): X\rightarrow Y$ and $(f_2,g_2): Y \rightarrow Z$ is given by the pullback diagram
\[
\xymatrix{
& & Z_1 \times_{Y} Z_2 \ar[dl]_{\pi_1} \ar[dr]^{\pi_2} & & \\
& Z_1 \ar[dl]_{f_1} \ar[dr]^{g_1} & & Z_2 \ar[dl]_{f_2} \ar[dr]^{g_2} & \\
X & & Y & & Z
}
\]
where the pullback is given as the subspace $Z_1 \times_{Y} Z_2 =\lbrace (z_1,z_2) \,|\, g_1(z_1)=f_2(z_2) \rbrace \subset 
Z_1\times Z_2$. It is equivalently the kernel of the map $g_1\circ \pi_1 - f_2\circ \pi_2 : Z_1 \times Z_2 \to Y$, where $\pi_1$ and $\pi_2$ denote the projections of $Z_1\times Z_2$ onto its factors.
\end{itemize}

\begin{rem}
$\Vect_\CC$ naturally embeds inside $\Span(\Vect_\CC)$. Indeed, given a linear map $f: X \to Y$, we associate to it the span
$$
\xymatrix{
	& X \ar[ld]_{\id_X} \ar[rd]^{f} & \\
	X & & Y
}
$$
We shall refer to these special spans as \emph{basic spans}. Note that composition of basic spans is preserved under this assignment, since the pullback of two spans $X \stackrel{\id_X}{\longleftarrow} X \stackrel{f}{\longrightarrow} Y$ and $Y \stackrel{\id_Y}{\longleftarrow} Y \stackrel{g}{\longrightarrow} Z$ is given by
\begin{equation} \label{diag:spanproduct}
\xymatrix{
& & X = X \times_Y Y \ar[dl]_{\id_X} \ar[dr]^{f} & & \\
& X \ar[dl]_{\id_X} \ar[dr]^{f} & & Y \ar[dl]_{\id_Y} \ar[dr]^{g} & \\
X & & Y & & Z
}
\end{equation}
Therefore, this span corresponds exactly to the basic span associated with the composition $g \circ f$.
\end{rem}

Finally, notice that any span $X \stackrel{f}{\longleftarrow} Z \stackrel{g}{\longrightarrow} Y$ such that $f$ is an isomorphism is equivalent to a basic span in a unique way. Indeed, we have the equivalence
$$
\xymatrix{
	& Z \ar[ld]_{f} \ar[rd]^{g} \ar@{--{>}}[dd]^{f} & \\
	X & & Y \\
	& X \ar[lu]^{\Id_X} \ar[ru]_{g \circ f^{-1}} & 
}
$$

\begin{rem}
    We also have a ``dual'' inclusion $(-)^*: \Vect_\CC^{\textup{op}} \to \Span(\Vect_\CC)$ by associating to a vector space $X$ its dual $X^*$. In this manner, a linear map $f: X \to Y$ embeds as the span
$$
\xymatrix{
	& Y^* \ar[ld]_{\Id_{Y^*}} \ar[rd]^{f^*} & \\
	Y^* & & X^*
}
$$
The image subcategory is thus exactly the one containing the spans $X \stackrel{f}{\longleftarrow} Z \stackrel{g}{\longrightarrow} Y$ such that $g$ is an isomorphism, and the original vector spaces are recovered by double dual $(f\circ g^{-1})^*: X^* \to Z^*$.
\end{rem}


\subsubsection*{Rigid monoidal structure}

The category $\Vect_\CC$ is a monoidal category with the \emph{direct sum} $\oplus$ of vector spaces. This monoidal structure is inherited by $\Span(\Vect_\CC)$ through direct sum of both objects and spans. Explicitly, given spans $X \stackrel{f}{\longleftarrow} Z \stackrel{g}{\longrightarrow} Y$ and $X' \stackrel{f'}{\longleftarrow} Z' \stackrel{g'}{\longrightarrow} Y'$, their direct sum is the span
$$
	\xymatrix{
& Z \oplus Z' \ar[dl]_{f \oplus f'} \ar[dr]^{g \oplus g'} & \\
X \oplus X' & & Y \oplus Y'
}
$$
Notice that $\oplus$ is both the product and coproduct in $\Vect_\CC$, so $(\Vect_\CC, \oplus)$ is a cartesian and cocartesian category. The monoidal unit of this structure is the zero vector space $0$.

Furthermore, observe that the monoidal category $(\Span(\Vect_{\CC}), \oplus)$ is rigid and, indeed, every object is self-dual (see Section \ref{sec:duality}).

\subsection{The Topological Quantum Field Theory}\label{sec:TQFT}

In the setting described in Sections \ref{sec:tangle-cat} and \ref{sec:span-cat}, the aim of this section is to prove that, fixed $t \in \CC^*$, there exists a braided monoidal functor
$$
\cA_t: \cT \longrightarrow \Span(\Vect_{\CC})
$$
computing representation varieties on the Alexander quandle $\AQ{t}$.

This assignment is defined as follows.
\begin{itemize} 
\item On objects, we set $\cA_t(\emptyset)= 0$ and $\cA_t(\epsilon) = \CC^{|\epsilon|}$ for any object $\epsilon$ of $\cT$. More abstractly, we define
$$
	\cA_t(\epsilon) = \R_{\AQ{t}}(\epsilon) := \Hom_{\Qdl}(Q_{\epsilon}, \AQ{t}),
$$
where $Q_{\epsilon}$ is the quandle associated to the oriented codimension two submanifold $[|\epsilon|] \times \{0\} \subset \RR^2$ given by the points of $\epsilon$ with the orientations given by the signs $\epsilon_i$. 
Its quandle $Q_{\epsilon}$ is the free quandle generated by these points, and $\As(Q(\epsilon^c))=\pi_1(\epsilon^c)$, where $\epsilon^c = \RR^2 - \left([|\epsilon|] \times \{0\}\right)$ is the plane punctured at the points of $\epsilon$. It follows that $\pi_1(\epsilon^c)=\Free{|\epsilon|}$, the free group in $|\epsilon|$ generators.

\item On morphisms, given any tangle $T: \epsilon \to \varepsilon$, we consider the representation variety
$$
\R_{\AQ{t}}(T):=\Hom_\Qdl(Q_T, \AQ{t}),
$$
where $Q_T$ is the quandle of the tangle $T$, which can be finitely presented using a planar projection, as in the case of knots and links. Using a Wirtinger presentation we may conclude that $\As(Q(T))=\pi_1(T^c)$, being $T^c = (\RR^{2} \times [0,1]) - T$ the complement of the tangle in $\mathbb{R}^2\times[0,1]$, following the notation used for knots. 

Since $\partial T^c = \overline{\epsilon^c} \sqcup \varepsilon^c$, we obtain inclusion maps $Q_{\epsilon}, Q_{\varepsilon} \hookrightarrow Q_T$ and therefore restrictions $i: \R_{\AQ{t}}(T) \to \R_{\AQ{t}}(\epsilon)$ and $j: \R_{\AQ{t}}(T) \to \R_{\AQ{t}}(\varepsilon)$. Thus, we set $\cA_t(T)$ to be the span
\[
\xymatrix{
& \R_{\AQ{t}}(T) \ar[dl]_{i} \ar[dr]^{j} & \\
\R_{\AQ{t}}(\epsilon) & & \R_{\AQ{t}}(\varepsilon)
}
\]
Notice that this is a span in $\Span(\Vect_\CC)$, since all these representation varieties are complex vector spaces as shown in Section \ref{sec:alexander-broader}.

\end{itemize}

\begin{rem}\label{rem:quandlegroupvarietiesiso}
By Remark \ref{rmk:quandle-alexander}, the quandle representation varieties $\R_{\AQ{t}}(T)$ and $\R_{\AQ{t}}(\epsilon)$ can be identified with the set of representations of their associated groups into $\agl$ where the parameter $t\in \CC^{\ast}$ remains constant, as defined in \eqref{eqn:agldescription}.
\end{rem}

\begin{prop}\label{prop:At-functor}
$\cA_t$ is a functor from $\cT$ to $\Span(\Vect_{\CC})$.
\end{prop}

\begin{proof}
Given any object $\epsilon$ of $\cT$ of length $n = |\epsilon|$, the identity tangle $\id_{\epsilon}$ is given by the union of intervals $[n]\times \{0\} \times [0,1] \subset \mathbb{R}^2 \times [0,1]$, that is, $n$ intervals with the same orientation. Its fundamental group is free in $n$ generators, so $\R_{\AQ{t}}(\id_{\epsilon}^c) = \CC^n$ and the restriction maps to both boundaries 
$\R_{\AQ{t}}(\epsilon^c) = \CC^n$ are the identity. Thus, the associated span $\cA_t(\id_{\epsilon}) = (\CC^n, \id_{\CC^n}, \id_{\CC^n})$ is the identity span.

 With respect to composition, consider two tangles $T: \epsilon_1 \to \epsilon_2$ and $T': \epsilon_2 \to \epsilon_3$, which induce spans $\R_{\AQ{t}}(\epsilon_1) \stackrel{i}{\longleftarrow} \R_{\AQ{t}}(T) \stackrel{j}{\longrightarrow} \R_{\AQ{t}}(\epsilon_2)$ and $\R_{\AQ{t}}(\epsilon_2) \stackrel{i'}{\longleftarrow} \R_{\AQ{t}}(T') \stackrel{j'}{\longrightarrow} \R_{\AQ{t}}(\epsilon_3)$. Considering the pushout in the quandle category, that is, the amalgamated product of $Q_T$ and $Q_{T'})$, we get that $Q_{T'\circ T}=Q_{T'}\ast_{Q_{\epsilon_2}} 
Q_T$ (see \cite{cattabriga}). Applying the continuous functor $\Hom_{\textbf{Qdl}}(-, \AQ{t})$ to the pushout diagram for $Q_{T'\circ T}$ induced by the inclusions, we obtain the pullback diagram
\[
\xymatrix{
& & \R_{\AQ{t}}(T'\circ T) \ \ar@{-->}[dl] \ar@{-->}[dr] & & \\
& \R_{\AQ{t}}(T) \ar[dl] \ar[dr]^{(i_1)_*} & & \R_{\AQ{t}}(T') \ar[dl]_{(i_2)_*} \ar[dr] & \\
\R_{\AQ{t}}(\epsilon_1) & & \R_{\AQ{t}}(\epsilon_2) & & \R_{\AQ{t}}(\epsilon_3)
}
\]
which is precisely the composition of the spans $\cA_t(T)$ and $\cA_t(T')$.
\end{proof}

\begin{rem}
Notice that, by the previous proof, the functor $\cA_t$ automatically satisfies all the relations of Remark \ref{rem:relations-tangles}.
\end{rem}

\begin{rem}
The quandles associated to any tangle and its source and target form a category, the \emph{bordered quandle category} $\mathcal{BQ}$, whose description  can be found in \cite{cattabriga}. The functor $BQ:\cT \longrightarrow \mathcal{BQ}$ that associates to any tangle its bordered quandle is called the \emph{fundamental quandle functor}. In this sense, we might state that the functor $\cA_t$ factorizes via $BQ$, representing bordered quandles in $\AQ{t}$, so that $\cA_t= \Hom(-,\AQ{t})\circ BQ$.
\end{rem}

Another direct consequence of the definition is the following.
\begin{prop}\label{prop:At-monoidal}
The functor $\cA_t: (\cT, \otimes) \to (\Span(\Vect_\CC), \oplus)$ is monoidal.
Given two tangles $T,T'\in \cT$, $\cA(T\otimes T')=\cA(T)\oplus \cA(T')$
\end{prop}
\begin{proof}
Given two objects $\epsilon, \varepsilon$ of $\cT$, we have that
\begin{align*}
	\cA_t(\epsilon \otimes \varepsilon) &= \Hom_{\Qdl}(Q_{\epsilon \otimes \varepsilon}, \AQ{t}) = \Hom_{\Qdl}(\textbf{FQ}_{|\epsilon|+|\varepsilon|}, \AQ{t}) \\
    & = (\AQ{t})^{|\epsilon|+|\varepsilon|}= (\AQ{t})^{|\epsilon|}\times (\AQ{t})^{|\varepsilon|}  \\    
  &= \Hom_{\Qdl}(Q_{\epsilon}, \AQ{t}) \oplus \Hom_{\Qdl}(Q_{\varepsilon}, \AQ{t}) = \cA_{t}(\epsilon) \oplus \cA_t(\varepsilon).
\end{align*}
Here, in the second equality, we used that the quandle $Q_{\epsilon\otimes\varepsilon}$ is the free quandle in $|\epsilon|+|\varepsilon|$ generators, so that for the free quandle $\textbf{FQ}_n$ on $n$ elements, $\Hom_{\Qdl}(\textbf{FQ}_n, \AQ{t})=\Hom_{\textbf{Set}}(n,\AQ{t})=(\AQ{t})^n$ from its universal property. The last two equalities can be interpreted in the category $\Vect_\CC$, since the underlying set of $\AQ{t}$ is $\CC$.

Analogously, given two tangles $T$ and $T'$, the tangle $T\otimes T'$ is obtained topologically as the connected sum along a disc, which is contractible. From the presentation of $Q_{T\otimes T'}$ in terms of arcs and relations, we obtain that 
$Q_{T\otimes T'}\cong Q_T\star Q_{T'}$, described in Example \ref{ex:quandle-free-product}. Then
$\R_{\AQ{t}}(T\otimes T')= \Hom_{\Qdl}(Q_T \star Q_{T'}, \AQ{t}) =\R_{\AQ{t}}(T)\oplus \R_{\AQ{t}}(T')$, and the corresponding restriction maps also split as a direct sum, so $\cA_t(T \otimes T') = \cA_t(T) \oplus \cA_t(T')$ as span.
\end{proof}

Additionally, we can also easily identify the object computed by this functor $\cA_t$ for endomorphisms of the unit. Indeed, if $T: \emptyset \to \emptyset$ is a tangle without endpoints, then we can see $T$ as a link in $\RR^3$ and, by its very definition, we have that $\cA_t(T)$ is the span
\[
\xymatrix{
& \R_{\AQ{t}}(T) \ar[dl]\ar[dr] & \\
0 & & 0
}
\]

Notice that the representation variety $\R_{\AQ{t}}(T)$ coincides with the representation variety of $T$ as a link in $S^3$, since compactifying by a point does not change the fundamental group or quandle. However, even more holds in this situation. Extending the discussion from Section \ref{sec:Alexmodulematrix} to tangles, we obtain a natural embedding $\R_{\AQ{t}}(T) \subset \CC^N$, where $N$ is the number of arcs in a planar projection of $T$. Representations of $\pi_1(T^c)$ into $\agl$, derived from a Wirtinger presentation, show again that  $\R_{\AQ{t}}(T)$ is the zero set of certain linear forms $\omega_1, \ldots, \omega_s: \CC^N \to \CC$. The determinant of these forms $\omega_1, \ldots, \omega_s$ measures the existence of non-abelian $\agl$-representations of $\pi_1(T^c)$, analogous to how the Alexander polynomial $\Delta_T (t)$ does for knots and links. In this sense, we shall say that $\cA_{t}$ computes the Alexander polynomial.

Putting together all this information, we get the main result of this section.

\begin{cor}
For any $t \in \CC^*$, 
$$
	\cA_t: \cT \to \Span(\Vect_\CC)
$$
is a monoidal braided functor $\cA_t $ which computes the Alexander polynomial.
\end{cor}


\subsection{TQFT for singular vector bundles} \label{sec:TQFTsingular}

In this section, we shall show that the TQFT developed in Section \ref{sec:TQFT} for each fibre over $t \in \CC^*$ of the Alexander fibration can be extended to a global TQFT of singular vector bundles over $\CC^*$.

Recall from Definition \ref{defn:singular-vector-bundle} that, given an algebraic variety $X$, a singular vector bundle $\pi: E \to X$ is the relative spectrum of the symmetric algebra of a coherent sheaf. This means that the variety $E$ is equipped with local trivializing injections 
$\pi^{-1}(U) \to U \times \mathbb{C}^n$, for $U \subset X$, that describe $E$ as a vector bundle with fibres of varying rank. These spaces form an abelian category $\SVB$ with singular vector bundles as objects and fibrewise linear maps as morphisms. In particular, $\SVB$ has pullbacks, which are given by fibered products. 
If we fix the base space $X$, the slice category $\SVB/X$ is actually isomorphic to the category of coherent sheaves over $X$ by the relative spectrum construction. However, we shall focus on singular vector bundles instead of coherent sheaves due to their nature resembling families of vector spaces parametrized by $X$. 

In this way, let us also consider the category $\Span(\SVB)$ of spans of singular vector bundles. As in the case of Section \ref{sec:span-cat}, the objects of $\Span(\SVB)$ are singular vector bundles, whereas a morphism between bundles $E \to X$ and $E' \to X'$ is an equivalence class of spans
$$
\xymatrix{
& E'' \ar[dl]_{(f, \tilde{f})} \ar[dr]^{(g, \tilde{g})} & \\
E & & E'
}
$$
where $(f, \tilde{f}): E'' \to E$ and $(g, \tilde{g}): E'' \to E'$ are singular vector bundle morphisms, and two spans are declared to be equivalent if there is an intertwining isomorphism between the middle bundles. Again, composition of spans is given by pullback and $\Span(\SVB)$ is a monoidal category with respect to fibrewise direct sum.

In this context, we will show that there exists a braided monoidal functor
$$
\cA: \cT \longrightarrow \Span(\SVB)
$$
representing $\agl$-representation varieties, and such that $\mathcal{A}_t: \cT \to \Span(\Vect_\CC)$ arises as its fibres. This assignment is defined as follows.
\begin{itemize} 
\item On objects, we set
$$
	\cA(\epsilon) = \R_{\agl}(\epsilon) := \Hom_{\textbf{Grp}}(\Gamma_\epsilon, \agl),
$$
where $\Gamma_\epsilon = \pi_1(\epsilon^c)$ is the fundamental group of the complement of the points of $[|\epsilon|]\times \{0\} \subset \RR^2$. 
Under the map $\agl \to \CC^*$, we see $\R_{\agl}(\epsilon)$ as a singular vector bundle over $\R_{\CC^*}(\epsilon)$ as in Section \ref{sec:alexander-broader}.

\item On morphisms, given any tangle $T: \epsilon \to \varepsilon$, we consider the representation variety
$$
\R_{\agl}(T)=\Hom_{\textbf{Grp}}(\Gamma_T, \agl),
$$
where $\Gamma_T = \pi_1(\RR^2 \times [0,1] - T)$ is tangle group of $T$. Again, the morphism $\R_{\agl}(T) \to \R_{\CC^*}(T)$ realizes $\R_{\agl}(T)$ as a singular vector bundle.

Since $\partial T^c = \overline{\epsilon^c} \sqcup \varepsilon^c$, we obtain inclusion maps $\Gamma_\epsilon, \Gamma_\varepsilon \rightarrow \Gamma_T$ and therefore singular vector bundle morphisms given by restriction maps $i: \R_{\agl}(T) \to \R_{\agl}(\epsilon)$ and $j: \R_{\agl}(T) \to \R_{\agl}(\varepsilon)$. Thus, we set $\cA(T)$ to be the span in $\Span(\SVB)$
\[
\xymatrix{
& \R_{\agl}(T) \ar[dl]_{i} \ar[dr]^{j} & \\
\R_{\agl}(\epsilon) & & \R_{\agl}(\varepsilon)
}
\]
\end{itemize}

Arguments completely analogous to those of Propositions \ref{prop:At-functor} and $\ref{prop:At-monoidal}$ show that the assignment $\mathcal{A}: \cT \to \Span(\SVB)$ is a monoidal functor. Furtheremore, for a tangle without endpoints $T: \emptyset \to \emptyset$, 
that is a link, we have that $\cA(L)$ is given by
\[
\xymatrix{
& \R_{\agl}(T) \ar[dl]\ar[dr] & \\
0 & & 0
}
\]
Here $\R_{\agl}(T)$ is seen as a vector bundle over $\R_{\CC^*}(T) = \CC^s$, where $s$ is the number of components of the link $T$. As explained in Section \ref{sec:alexander-broader}, the singular loci of this vector bundle are given by the subsets 
$Z_T^k$ in (\ref{eq:Alexander-stratification-link}), and in particular $Z_T^1$ is given by
multivariable Alexander polynomials $\Delta_T(t_1, \ldots, t_s)$. In this sense, we shall say that $\cA$ computes the multivariable Alexander polynomial.

Finally, observe that, for any $t \in \CC^*$, there exists a natural transformation
$$
    \cA_t \Longrightarrow \cA
$$
given by `inserting the $t$-fibre'. First, observe that any vector space can be seen as a singular vector bundle over the singleton variety $\star$, giving an inclusion $\Vect_\CC \hookrightarrow \SVB$ and a fortiori $\Span(\Vect_\CC) \hookrightarrow \Span(\SVB)$. This allows us to see $\cA_t$ as a functor $\cA_t: \cT \to \Span(\SVB)$. 

Now, fix an object $\epsilon$ of $\cT$. Recall that $\AQ{t} \subset \Conj(\agl)$ as the subquandle of matrices with diagonal elements equal to $t$, so we have an inclusion
\begin{align*}
    \tilde{f}_t: \R_{\AQ{t}}(\epsilon) = \Hom_{\Qdl}(Q_{\epsilon}, \AQ{t}) \hookrightarrow &\,
\Hom_{\Qdl}(Q_{\epsilon}, \Conj(\agl)) \\ & \cong \Hom_{\textbf{Grp}}(\As(Q_{\epsilon}), \agl) 
= \R_{\agl}(\epsilon),
\end{align*}
where we have used that $\As(Q_{\epsilon}) = \pi_1(\epsilon^c)$ as groups. 

Observe that if we consider the inclusion of the singleton set $f_t: \star \to \R_{\CC^*}(\agl) = (\CC^*)^{|\epsilon|}$ given by $f_t(\star) = (t, t, \ldots, t)$, we can see $\R_{\AQ{t}}(\epsilon)$ as a (singular) vector bundle over $\star$ and $(f_t, \tilde{f}_t): \R_{\AQ{t}}(\epsilon) \to \R_{\agl}(\epsilon)$ as a singular vector bundle morphism. Therefore, if we set $\tau_t(\epsilon) = (f_t, \tilde{f}_t): \R_{\AQ{t}}(\epsilon) \to \R_{\agl}(\epsilon)$ to be this inclusion, it is straightforward to check that these morphisms commute with spans and thus define a natural transformation $\tau_t: \cA_t \Rightarrow \cA$.

\begin{rem}
    The construction of the functor $\cA$ does not depend on orientations of the tangle, and thus can be promoted to a functor on the category of non-oriented tangles. In sharp contrast, $\cA_t$ does depend on the chosen orientation of the tangle, since the tangle quandle requires it to be defined. This can also be seen from the fact that $H_1(T^c)$ is generated by the meridians and thus identifying it with $\ZZ^s$ (and thus identifying $\R_{\CC^*}(T) \cong (\CC^*)^s$) is equivalent to fixing an orientation on $T$.
\end{rem}


\section{Explicit calculations and Burau representations} \label{sec:explicitcalculations}

\subsection{Action on the generators of $\cT$} \label{sec:generatorsdescription} Let us describe the functor $\cA_t$ explicitly in terms of the generators of the tangle category. Thanks to the orientations in the tangles and their boundaries, we have a natural presentation of the associated fundamental groups as follows.

\begin{itemize}
	\item Objects. Given an object $\epsilon$ of $\cT$ of length $n = |\epsilon|$, we have $\Gamma_\epsilon =\pi_1(\epsilon^c) = \Free{n}$, with generators looping counterclockwise (as seen `from above') for positive oriented points $+$, and loop clockwise for negative oriented points $-$, as shown in Figure \ref{fig:generators-points}.
	
\begin{figure}[h]
    \centering
    \includegraphics[width=0.5\linewidth]{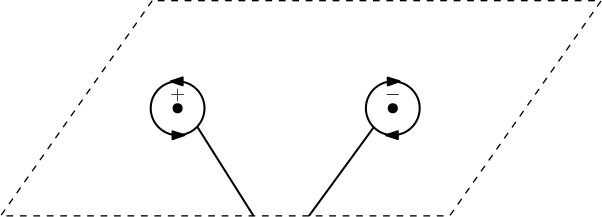}
    \caption{Generators of $\Gamma_\epsilon = \pi_1(\RR^2 - \epsilon)$.}
    \label{fig:generators-points}
\end{figure}

    In the same vein, removing the loops around the points and considering only the paths to a neighborhood of $\epsilon$, we have $Q_\epsilon = \FQ(n)$, so we have a natural identification $\cA_t(\epsilon) = \R_{\AQ{t}}(\epsilon^c) = \CC^n$.
	
	\item Morphisms. Given a tangle $T: \epsilon \to \varepsilon$, we consider generators of $\Gamma_T = \pi_1(T^c)$ looping around the strands with the orientation given by the `right-hand rule', as shown in Figure \ref{fig:generators-tangles}.
	
\begin{figure}[h]
    \centering
    \includegraphics[width=0.2\linewidth]{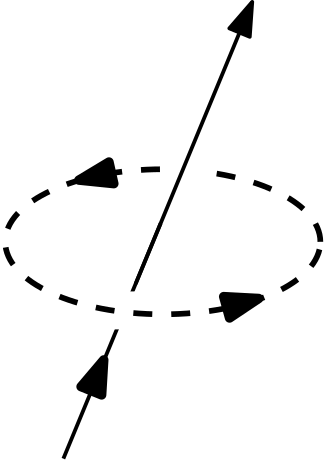}
    \caption{Generators of $\Gamma_T = \pi_1(\RR^2 \times [0,1] - T)$.}
    \label{fig:generators-tangles}
\end{figure}

    Analogously, this also defines generators of the quandle $Q_T$, leading to a description of $\R_{\AQ{t}}(T)$.
\end{itemize}

\subsubsection*{Braidings}

Let us consider the tangle $X_+: (+, +) \to (+, +)$, as depicted in Figure \ref{fig:X-plus-generators}.

\begin{figure}[h]
    \centering
    \includegraphics[width=0.23\linewidth]{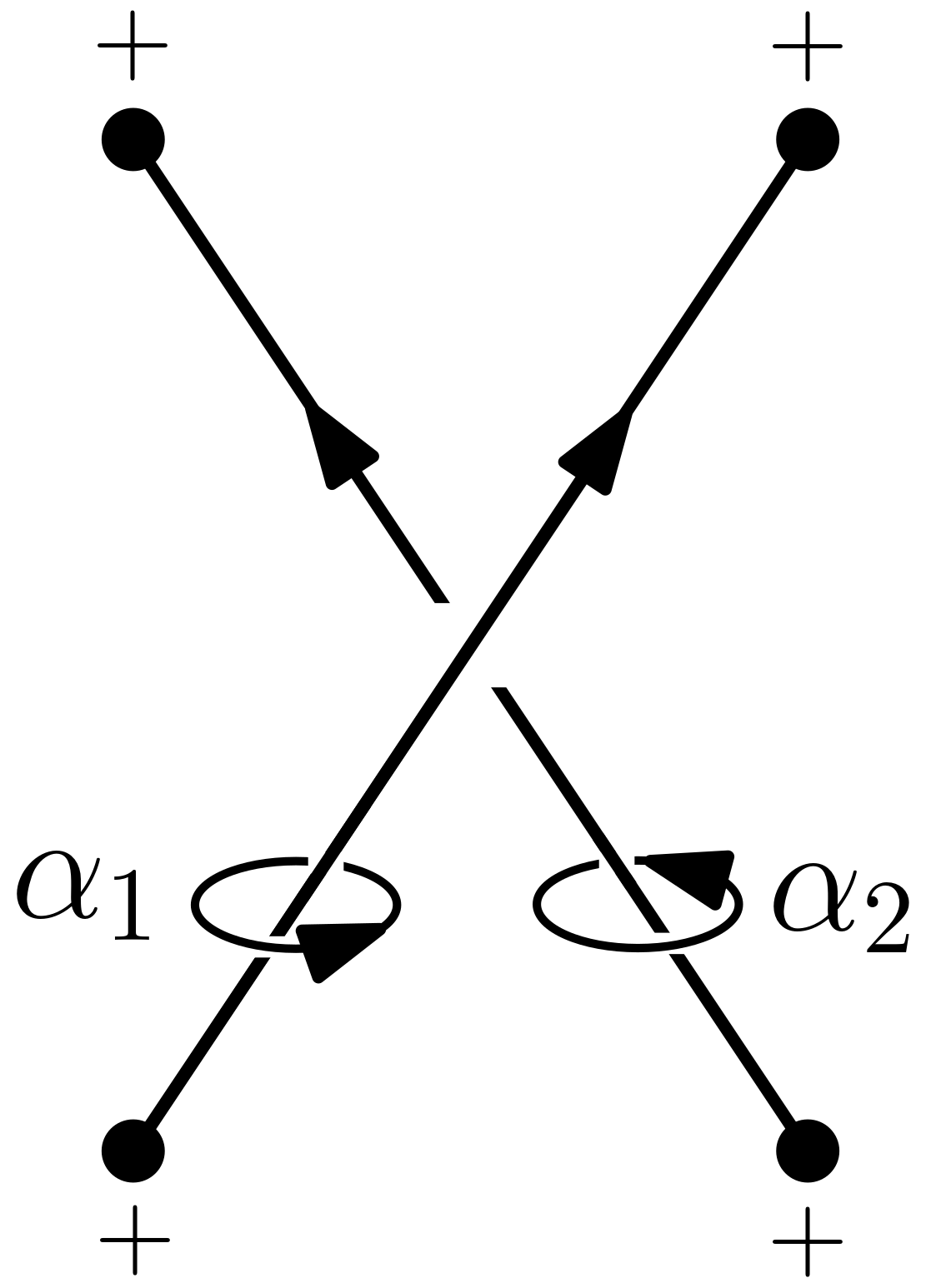}
    \caption{The tangle $X_+: (+, +) \to (+, +)$ and the generators of $\Gamma_{X_+} = \pi_1(\RR^2 \times [0,1] - X_+)$.}
    \label{fig:X-plus-generators}
\end{figure}

Using the generators $\pi_1(X_+^c) = \langle \alpha_1, \alpha_2\rangle \cong \Free{2}$ depicted in Figure \ref{fig:X-plus-generators}, we have that the restriction maps are the following.
$$
\begin{array}{rcl}
	\textrm{Incoming boundary:} & \Free{2} = \pi_1({(+,+)^c}) \to \Free{2} = \pi_1(X_+^c), & (\beta_1, \beta_2) \mapsto (\alpha_1, \alpha_2), \\
	\textrm{Outgoing boundary:} & \Free{2} = \pi_1({(+,+)^c}) \to \Free{2} = \pi_1(X_+^c), & (\beta_1, \beta_2) \mapsto (\alpha_1\alpha_2\alpha_1^{-1}, \alpha_1) = (\alpha_1 \triangleright \alpha_2, \alpha_1).
\end{array}
$$
Therefore, at the level of the quandle representation varieties, the induced span $\cA_t(X_+): \CC^2 \to \CC^2$ is given by
\[
\xymatrix{
& \CC^2 \ar[ld]_{\id} \ar[rd]^{f_+}& \\
\CC^2 & & \CC^2
}
\]
where $f_+(a_1, a_2) = (a_1 \triangleright_t a_2, a_1)=((1-t)a_1 + ta_2, a_1)$. In particular, this shows that the braiding induced in $\Span(\Vect_\CC)$ by $\cA_t$ is precisely $\cA_t(X_+)$. Notice that this braiding is not symmetric in general, since $\cA_t(X_+)^2$ is the span $\CC^2 \stackrel{\id}{\longleftarrow} \CC^2 \stackrel{f_+^2}{\longrightarrow} \CC^2$, where $f_+^2(a_1, a_2) = (((1-t)^2 + t)a_1 + t(1-t)a_2, (1-t)a_1 + ta_2)$, which is not symmetric except for $t = 1$.

Analogously for the tangle $X_-: (+, +) \to (+, +)$, we have that the inclusion of the incoming boundary is the identity $\id: \Free{2} = \pi_1({(+,+)^c}) \to \Free{2} = \pi_1(X_-^c)$, but the inclusion of the outgoing boundary is $\Free{2} = \pi_1({(+,+)^c}) \to \Free{2} = \pi_1(X_-^c)$, $(\beta_1, \beta_2) \mapsto (\alpha_2, \alpha_2^{-1}\alpha_1\alpha_2) = (\alpha_2, \alpha_1 \triangleleft \alpha_2)$. Therefore, the induced span is
\[
\xymatrix{
& \CC^2 \ar[ld]_{\id} \ar[rd]^{f_-}& \\
\CC^2 & & \CC^2
}
\]
where $f_-(a_1, a_2) = (a_2, a_1 \triangleleft_t a_2)=(a_2, (1-t^{-1})a_2+t^{-1}a_1)$. 

\subsubsection*{Evaluation and coevaluation maps}\label{sec:duality}

In this section, we compute the image of the maps $\curvearrowright: (+, -) \to \emptyset$ and $\curvearrowleft: (-, +) \to \emptyset$, as well as $\upcurvearrowright: \emptyset \to (+, -)$ and $\upcurvearrowleft: \emptyset \to (-, +)$.

First, for $\curvearrowright$, observe that $\pi_1(\curvearrowright^c) = \langle \alpha \rangle = \ZZ$, and then the inclusion map $\pi_1({(+,-)^c}) = \Free{2} \to \pi_1(\curvearrowright^c) = \ZZ$ is given by $\beta_1,\beta_2 \mapsto \alpha$, as shown in Figure \ref{fig:dual}.

\begin{figure}[h]
    \centering
    \includegraphics[width=0.25\linewidth]{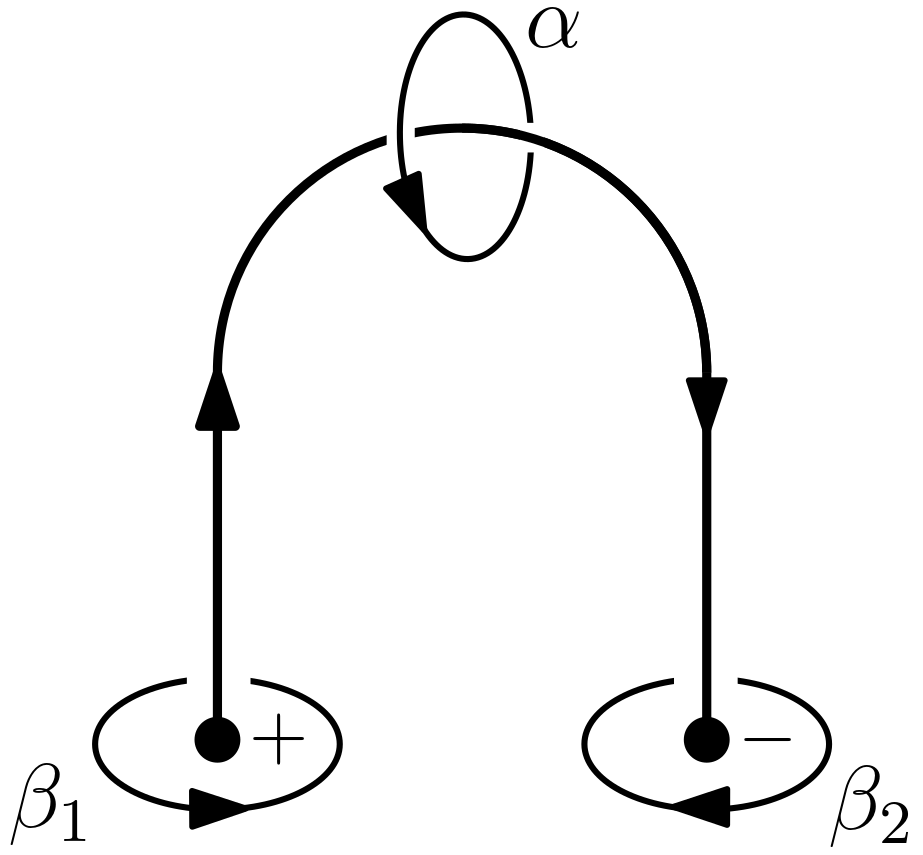}
    \caption{The tangle $\curvearrowright: (+, -) \to \emptyset$ and the generators of $\Gamma_{\curvearrowright} = \pi_1(\RR^2 \times [0,1] - \curvearrowright)$.}
    \label{fig:dual}
\end{figure}

Therefore, the associated span $\cA_t(\curvearrowright)$ is
\[
\xymatrix{
& \CC \ar[ld]_{\Delta} \ar[rd]& \\
\CC^2 & & 0
}
\]
where $\Delta(a) = (a, a)$ is the diagonal map. Proceeding analogously for the map $\curvearrowleft$, we find that $\cA_t(\curvearrowleft) = \cA_t(\curvearrowright)$. Additionally, both $\cA_t(\upcurvearrowright)$ and $\cA_t(\upcurvearrowright)$ are the span
\[
\xymatrix{
& \CC \ar[rd]^{\Delta} \ar[ld]& \\
0 & & \CC^2
}
\]

Following the previous construction, it is in general possible to obtain the evaluation and coevaluation maps attached to any object $\epsilon = (\epsilon_1, \ldots, \epsilon_n)$ of $\cT$. Indeed, the evaluation is a tangle $\textup{ev}_\epsilon: \epsilon \otimes \epsilon^{*} \to \emptyset$, where $\epsilon^* = (-\epsilon_n, \ldots, -\epsilon_1)$ is the dual object, given by `bending the identity', as in Figure \ref{fig:dual-general}.

\begin{figure}[h]
    \centering
    \includegraphics[width=0.35\linewidth]{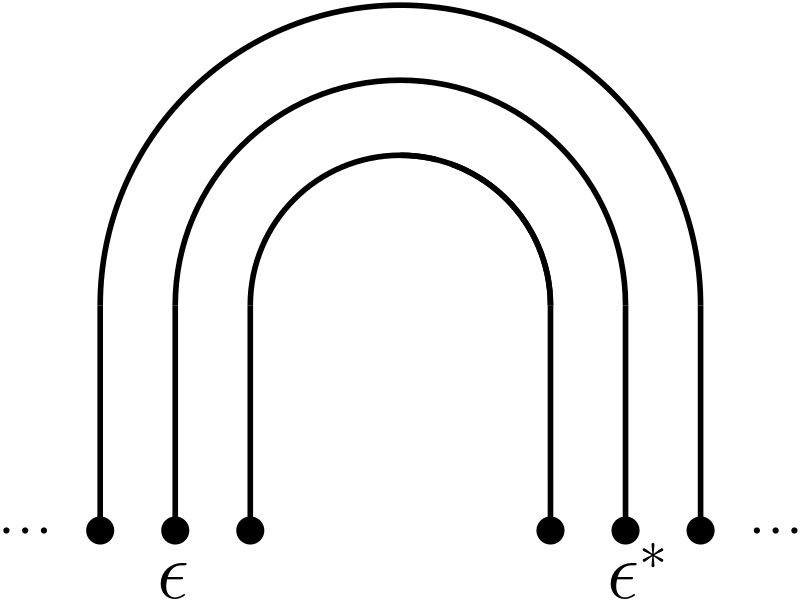}
    \caption{The evaluation tangle $\textup{ev}_\epsilon: \epsilon \otimes \epsilon^{*} \to \emptyset$.}
    \label{fig:dual-general}
\end{figure}

Observe that there is a natural identification $\R_{\AQ{t}}(\textup{ev}_\epsilon) = \R_{\AQ{t}}(\epsilon)$ and the span $\cA_t(\textup{ev}_\epsilon)$ is given by
\[
\xymatrix{
& \R_{\AQ{t}}(\epsilon) \ar[ld]_{\Delta} \ar[rd]& \\
\R_{\AQ{t}}(\epsilon)^{\oplus 2} & & 0
}
\]
where the morphism $\Delta: \R_{\AQ{t}}(\epsilon) \to \R_{\AQ{t}}(\epsilon)^{\oplus 2}$ is given by $\Delta(z) = (z, z^*)$, with $z^* = (z_n, \ldots, z_1)$ if $z = (z_1, \ldots, z_n)$. Observe that this morphism makes sense with since $\cA_t( \epsilon \otimes \epsilon^{*}) = \cA_t(\epsilon) \oplus \cA_t(\epsilon^{*}) = \cA_t(\epsilon)^{\oplus 2}$.

Analogously for the coevaluation tangle $\textup{coev}_\epsilon: \emptyset \to \epsilon^* \otimes \epsilon$, we have that $\cA_t(\textup{coev}_\epsilon)$ is the span 
\[
\xymatrix{
& \R_{\AQ{t}}(\epsilon) \ar[rd]^{\Delta} \ar[ld]& \\
0 & & \R_{\AQ{t}}(\epsilon)^{\oplus 2}
}
\]


\subsection{The braid subcategory} \label{sec:braids}

The category $\cT$ of tangles admits a distinguished subcategory $\cB$, called the \emph{category of braids}. It contains the objects of the form $(+, \ldots, +)$ of $\cT$, that we will simply denote by its length $n$. As morphisms, we only allow tangles always moving upwards i.e.\ those tangles whose tangent line is never horizontal, as shown in Figure \ref{fig:braid}.

\begin{figure}[h]
    \centering
    \includegraphics[width=0.2\linewidth]{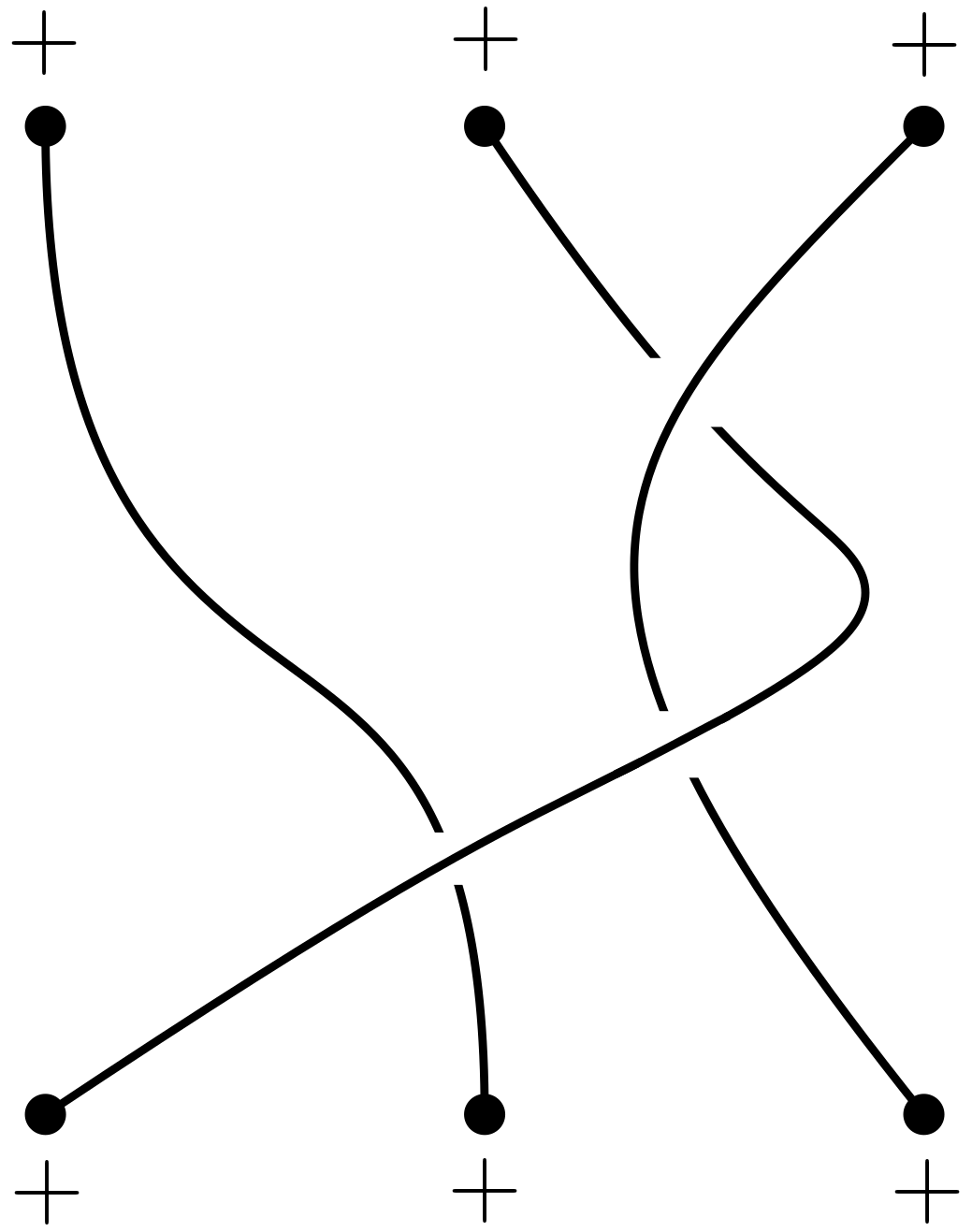}
    \caption{Abstract representation of a braid.}
    \label{fig:braid}
\end{figure}

In other words, we have that $\Hom_{\cB}(n, n) = \mathrm{B}_n$ is the Artin braid group, but $\Hom_{\cB}(n,m) = 0$ for $n \neq m$.

A crucial property is that, for braids, the span that we get has a very easy description. 

\begin{prop}
Given a braid $B$, the span $\cA_t(B)$ can be taken to be a basic span. 
\end{prop}

\begin{proof}
If $B: n \to n$ is a braid, then the fundamental group of its complement $\Gamma_B = \pi_1(B^c)$ is generated by loops around the strands near the incoming boundary. Hence, if we take these generators coinciding with those corresponding to $\G_n =\pi_1(n^c)$, we have that the inclusion maps for both the fundamental group $\G_n \to \G_B$ and the quandle $Q_n \to Q_B$ for the incoming boundary is the identity. Hence, the span $\cA_t(B)$ has the form
\[
\xymatrix{
& \R_{\AQ{t}}(n) \ar[ld]_{\id} \ar[rd]^{f} & \\
\R_{\AQ{t}}(n) & & \R_{\AQ{t}}(n)
}
\]
for a certain map $f: \R_{\AQ{t}}(n) \to \R_{\AQ{t}}(n)$ (which is actually an isomorphism). Therefore, the span $\cA_{t}(B)$ is basic.
\end{proof}

\begin{cor}\label{cor:TQFT-braids}
The functor $\cA_t: \cT \to \Span(\Vect_\CC)$ factorizes through $\cB$ to give a monoidal braided functor
$$
	\cA_t|_{\cB}: \cB \longrightarrow \Vect_\CC.
$$
In other words, the following diagram commutes
\[
\xymatrix{
\cT \ar[rr]^{\cA_t} & & \Span(\Vect_\CC) \\
\cB \ar[rr]_{\cA_t|_{\cB}}\ar@{_{(}-{>}}[u] & & \Vect_\CC\ar@{_{(}-{>}}[u]
}
\]
\end{cor}

\begin{rem}
In the language of quantization of invariants, the functor $\cA_t$ gives rise to a Topological Quantum Field Theory on $\cB$ quantizing the Alexander polynomial.
\end{rem}

According to Corollary \ref{cor:TQFT-braids}, a braid $B: n \to n$ gives rise to a linear map $\cA_t(B): \cA_t(n) \to \cA_t(n)$, which is explicitly the restriction map
$$
	\cA_t(B): \R_{\AQ{t}}(n) = \R_{\AQ{t}}(B^c) = \CC^n \longrightarrow \R_{\AQ{t}}(n) = \CC^n.
$$
It turns out that the map above is a well-studied map in the theory of braids, known as the Burau representation of $\mathrm{B}_n$ \cite{burau1935zopfgruppen}.

\begin{prop}
The map
$$
	\cA_t|_{\mathrm{B}_n}: \Hom_{\cB}(n,n) = \mathrm{B}_n \longrightarrow \GL{n}(\CC)
$$
coincides with the Burau representation of the braid group $\mathrm{B}_n$.
\end{prop}

\begin{proof} 
Recall that $\mathrm{B}_n$ is generated by the simple crossings $e_i=\id\otimes \id \otimes \ldots \otimes X_+ \otimes \ldots \otimes \id$, where $X_+$ is placed in the $i$-th factor for $i = 1, \ldots, n-1$. Therefore, for any $i$, we have 
$$
\cA_t(e_i)=\cA_t(\id)\oplus\cA_t(\id)\oplus \ldots \oplus \cA_t(X_+) \oplus \ldots \oplus \cA_t(\id)=\left(\begin{array}{cccc}
\id_{i-1} & 0 & 0 & 0 \\
0 & 1-t & t & 0 \\
0 & 1 & 0 & 0 \\
0 & 0 & 0 & \id_{n-i-1}
\end{array}\right).
$$
This is exactly the matrix associated to the generators of $\mathrm{B}_n$ in the Burau representation, as it is typically defined. The fact that $\cA_t|_{\mathrm{B}_n}: \mathrm{B}_n \to \GL{n}(\CC)$ is a group homomorphism follows directly from the functoriality of $\cA_t|_{\cB}$.
\end{proof}

Given an endomorphism $f: \CC^n \to \CC^n$ (for instance, $f = \cA_t(B): \CC^n \to \CC^n$ the Burau representation for a braid $B \in \mathrm{B}_n$), we define the \emph{span trace} of $f$ as the span
$$
	\textup{STr}(f) = \cA_t(\textup{ev}_n) \circ (f \oplus \cA_t(\id_{n})) \circ \cA_t(\textup{coev}_n),
$$
which is an endospan of $0$.

\begin{prop}\label{prop:span-trace}
For any $f: \CC^n \to \CC^n$, its span trace $\textup{STr}(f)$ is the span
\[
\xymatrix{
& \Fix(f) \ar[ld] \ar[rd]& \\
0 & & 0
}
\]
where $\Fix(f) = \{v \in \CC^n \,|\, f(v) = v\}$ is the fix locus of $f$.
\end{prop}

\begin{proof}
The proof is just a straightforward calculation. The span $f \oplus \cA_t(\id_{n})$ is given by
\[
\xymatrix{
& \CC^{2n} \ar[ld]_{\id_{2n}} \ar[rd]^{f \oplus \id_n}& \\
\CC^{2n} & & \CC^{2n}
}
\]
Therefore, when we pre-compose it with the coevaluation $\cA_t(\textup{coev}_n): 0 \stackrel{}{\longleftarrow} \CC^n \stackrel{\Delta}{\longrightarrow} \CC^{2n}$, we get the span
\[
\xymatrix{
& & \CC^n \ar@{-->}[dl]_{\id_{2n}} \ar@{-->}[dr]^{\Delta} & & \\
& \CC^n \ar[dl] \ar[dr]^{\Delta} & & \CC^{2n} \ar[dl]_{\id_{2n}} \ar[dr]^{f \oplus \id_n} & \\
0 & & \CC^{2n} & & \CC^{2n}
}
\]
Thus, post-composing this span with the evaluation  $\cA_t(\textup{ev}_n): \CC^{2n} \stackrel{\Delta}{\longleftarrow} \CC^n \stackrel{}{\longrightarrow} 0$, we get
\[
\xymatrix{
& & \CC^n \oplus_{\CC^{2n}} \CC^n \ar@{-->}[dl] \ar@{-->}[dr] & & \\
& \CC^n \ar[dl] \ar[dr]^{(f \oplus \id_n) \circ \Delta} & & \CC^{n} \ar[dl]_{\Delta} \ar[dr] & \\
0 & & \CC^{2n} & & 0
}
\]
where the top space is
\begin{align*}
	\CC^n \oplus_{\CC^{2n}} \CC^n &= \left\{(v_1, v_2) \in \CC^n \oplus \CC^n \,\,|\,\, (f \oplus \id_n) \circ \Delta(v_1) = \Delta(v_2) \right\} \\
    & = \left\{(v_1, v_2) \in \CC^n \oplus \CC^n \,\,|\,\, f(v_1) = v_2, \, v_1 = v_2 \right\} =  \left\{v \in \CC^n  \,\,|\,\, f(v) = v \right\} = \Fix(f),
\end{align*}
as we wanted to prove.
\end{proof}

Recall that given a braid $B \in \mathrm{B}_n$, its \emph{closure link} $\overline{B}$ is the link obtained by joining all the upper endpoints of $B$ with the bottom endpoints, in the same order, as in Figure \ref{fig:closure-link}.

\begin{figure}[h]
    \centering
    \includegraphics[width=0.25\linewidth]{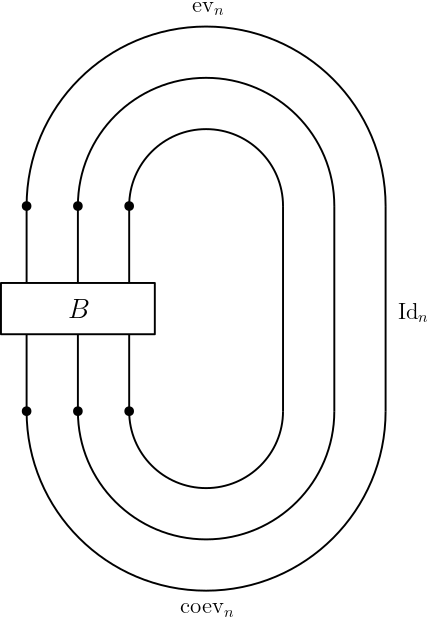}
    \caption{Closure link of a braid $B$.}
    \label{fig:closure-link}
\end{figure}

In other words, as a tangle we have that
\begin{equation}\label{eq:closure-braid}
	\overline{B} = \textup{ev}_n \circ (B \oplus \id_{n}) \circ \textup{coev}_n.
\end{equation}
Then, a direct consequence of Proposition \ref{prop:span-trace} is the following.

\begin{cor}
The Alexander module of the closure link of a braid is the coordinate ring of the fix locus of its Burau representation.
\end{cor}

\begin{cor}\label{cor:Burau}
Given a braid $B \in \mathrm{B}_n$, let $\cA_t(B): \CC^n \to \CC^n$ be its Burau representation. Consider the $n \times n$-matrix $ \Id - \cA_t(B)$, and let $(\Id - \cA_t(B))_0$ denote the $(n-1) \times (n-1)$-matrix given by removing its first row and column. Then, the Alexander polynomial of the closure link $\overline{B}$ of the braid $B$ is given by
$$
	\Delta_{\overline{B}}(t) = \det\,(\Id - \cA_t(B))_0.
$$
\end{cor}

\begin{proof}
By (\ref{eq:closure-braid}) and Proposition \ref{prop:span-trace}, we have that $\R_{\AQ{t}}(\overline{B}) = \Fix(\cA_{t}(B)) = \ker (\id - \cA_t(B))$. But, by Theorem \ref{thm:Alexander-coherent}, we have that the Alexander module $M_{\overline{B}}$ coincides with the module of regular functions on $\R_{\agl}(\overline{B})$, whose fibre over $t \in \CC^*$ is $\R_{\AQ{t}}(\overline{B})$. Hence $\ker(\id - \cA_t(B))$ is the vanishing ideal of $\R_{\agl}(\overline{B})$ and thus $\id - \cA_t(B)$ is a presentation matrix for the Alexander module. Therefore, any $(n-1)\times (n-1)$-minor of this matrix leads to the Alexander polynomial of $\overline{B}$.
\end{proof}

\begin{rem}
    In Corollary \ref{cor:Burau}, any other column or row can be removed to compute the determinant. However, since the Alexander polynomial is only determined up to a factor $\pm t^k$, it turns out that the resulting polynomial will differ by one of these factors. Making the choice of removing the first row and column leads to a polynomial with positive constant term, which coincides with the usual normalization of the Alexander polynomial.
\end{rem}

\begin{rem}
The automorphism $\cA_t(B): \CC^n \to \CC^n$ always has a fixed line. Indeed, identifying $\Fix(\cA_t(B)) = \R_{\AQ{t}}(\overline{B})$, we observe that any representation variety of a knot contains the abelian representations $x_i \mapsto a$ for a fixed $a \in \AQ{t}$ and every generator $x_i$ of $Q_{\overline{B}}$. In this manner, $\R_{\AQ{t}}(\overline{B})$ is at least one-dimensional for every $t \in \CC^*$, so the Burau representation $\cA_t(B)$ contains a fixed line (i.e.\ that generated by $(1, \ldots, 1)$). This implies that $\det (\Id - \cA_t(B))= 0$ for all $t \in \CC^*$, and hence one must consider smaller minors to extract meaningful information.

Furthermore, the above argument also shows that the Burau representation $\cA_t: B_n \to \GL{n}(\CC)$ is always a reducible representation, because it contains an invariant line. For this reason, it is customary in the literature to consider the so-called reduced Burau representation $\cA_t^{\textrm{red}}: B_n \to \GL{n-1}(\CC)$ by restricting $\cA_t$ to a certain subspace. Keeping track of the change of basis performed to get the reduced Burau representation, we get another formula for the Alexander polynomial that is typically exhibited in the literature
$$
	\Delta_{\overline{B}}(t) = \frac{1-t}{1-t^n} \det\,(\Id - \cA_t^{\textrm{red}}(B)).
$$
\end{rem}

\bibliographystyle{abbrv}
\bibliography{AGL1Alexander.bib}
\end{document}